\documentclass[11pt,reqno]{amsart}

\usepackage[utf8]{inputenc}
\usepackage[T1]{fontenc}
\usepackage[english]{babel}
\usepackage{lmodern}
\usepackage{microtype}
\usepackage{mathtools}
\usepackage{amssymb}
\usepackage{tikz-cd}
\usepackage[margin=1.15in]{geometry}
\usepackage[colorlinks=true,unicode,linkcolor=blue,citecolor=blue,urlcolor=blue]{hyperref}

\title[$\pi_*$-\'etale $\mathbb E_k$-algebras]{A rigidity theorem for $\pi_*$-\'etale $\mathbb E_k$-algebras}
\author{Xiansheng Li}
\address{Department of Mathematics, Southern University of Science and Technology, Shenzhen, China}
\email{12531032@mail.sustech.edu.cn}
\date{}

\newcommand{\Z}{\mathbb{Z}}
\newcommand{\E}{\mathbb{E}}
\DeclareMathOperator{\op}{op}
\DeclareMathOperator{\Sp}{Sp}
\DeclareMathOperator{\Mod}{Mod}
\DeclareMathOperator{\BMod}{BMod}
\DeclareMathOperator{\Map}{Map}
\DeclareMathOperator{\LMod}{LMod}
\DeclareMathOperator{\Psh}{Psh}
\DeclareMathOperator{\Alg}{Alg}
\DeclareMathOperator{\CAlg}{CAlg}
\DeclareMathOperator{\Cat}{Cat}
\DeclareMathOperator{\Fun}{Fun}

\numberwithin{equation}{section}
\theoremstyle{plain}
\newtheorem{theorem}{Theorem}[section]
\newtheorem{lemma}[theorem]{Lemma}
\newtheorem{proposition}[theorem]{Proposition}
\newtheorem{corollary}[theorem]{Corollary}

\theoremstyle{definition}
\newtheorem{construction}[theorem]{Construction}
\newtheorem{definition}[theorem]{Definition}
\newtheorem{example}[theorem]{Example}

\theoremstyle{remark}
\newtheorem{remark}[theorem]{Remark}
\newtheorem{warning}[theorem]{Warning}

\begin{document}

\begin{abstract}
We prove a rigidity theorem for $\pi_*$-\'etale $\mathbb E_k$-algebras over an $\mathbb E_{k+1}$-ring spectrum: the category of $\pi_*$-\'etale extensions of an $\mathbb E_k$-algebra is identified with the ordinary category of \'etale Dirac algebras over its graded homotopy Dirac ring. The proof develops a relative obstruction theory of Goerss--Hopkins type in synthetic spectra, including an $I$-complete version. As an application, the completed obstruction theory constructs the $I_n$-complete $\mathbb E_3$-$MU_{(p)}$-algebra realization of the Lubin--Tate theory, hence an $\mathbb E_4$-orientation $MU_{(p)}\to E_n$.
\end{abstract}

\maketitle

\section{Introduction}

\subsection{Dirac rings and homotopy theory}
The homotopy groups of a commutative ring spectrum satisfy the Koszul sign rule
\[
xy=(-1)^{|x||y|}yx.
\]
We study extensions that are \'etale on these graded rings, using the terminology of Dirac rings from Hesselholt and Pstr\k{a}gowski \cite{DG1}.

\begin{definition}\label{defdirac}
    A Dirac ring is a commutative algebra object in the symmetric monoidal (1-)category $\operatorname{grAb} = \Fun(\Z,\operatorname{Ab})$ of graded abelian groups with the symmetry isomorphism \[
x \otimes y \mapsto (-1)^{\operatorname{deg}(x)\operatorname{deg}(y)} y \otimes x.
\]
\end{definition}

\begin{remark}
    The category $\Fun(\Z,\Sp)$ of graded spectra inherits a symmetric monoidal structure from $\Sp$ by Day convolution (see \cite[Section~2.2.6]{HA}). The Beilinson $t$-structure (see \cite[Section~2.1]{DG1}) is compatible with this monoidal structure. The abelian category $\operatorname{grAb}$ identifies with the heart of this $t$-structure and inherits its symmetric monoidal structure from $\Fun(\Z,\Sp)$.
\end{remark}

\subsection{The main result}
\begin{definition}\label{defdiret}
    A map $f\colon A \rightarrow B$ of commutative rings is called an \'etale extension if
\begin{enumerate}
    \item The commutative ring $B$ is a finitely presented $A$-algebra.
    \item The map $f$ is formally \'etale, i.e., for any solid commutative square of commutative rings
    \[
\begin{tikzcd}
	A & C \\
	B & {C/I}
	\arrow[from=1-1, to=1-2]
	\arrow["f"', from=1-1, to=2-1]
	\arrow[from=1-2, to=2-2]
	\arrow[dashed, from=2-1, to=1-2]
	\arrow[from=2-1, to=2-2]
\end{tikzcd},
\] where the right vertical map is the projection onto the quotient by a nilpotent ideal $I \subset C$, there exists a unique dashed arrow making the whole diagram commute.
\end{enumerate}
\end{definition}
Lurie extended the notion of \'etaleness to ring spectra \cite[Definitions~7.5.0.4 and~7.5.1.4]{HA}:

\begin{definition}
A morphism $\phi\colon A \rightarrow B$ of $\mathbb{E}_k$-ring spectra is said to be $\pi_0$-\'etale if the following two conditions hold:
\begin{enumerate}
    \item The induced ring homomorphism $\pi_0(\phi)\colon \pi_0(A) \rightarrow \pi_0(B)$ is \'etale in the sense of Definition \ref{defdiret}.
    \item The map $\phi$ exhibits $B$ as a flat $A$-module. Namely, the associated map $\pi_n(A) \otimes_{\pi_0(A)} \pi_0(B) \rightarrow \pi_n(B)$ is an equivalence.
\end{enumerate}
\end{definition}

\begin{remark}
    In \cite[Section~7.5]{HA} the maps satisfying the above conditions are simply called  \'etale. We call these maps $\pi_0$-\'etale to distinguish them from the $\pi_*$-\'etale maps defined below.
\end{remark}

For an odd prime $p$, the inclusion of the periodic Adams summand into $p$-complete complex $K$-theory,
\[
L_p\longrightarrow KU_p,
\]
is a finite Galois extension, but it is not flat in the sense required for $\pi_0$-\'etaleness. Its full graded homotopy-ring map is nevertheless \'etale in the Dirac sense introduced below.

\begin{definition}[Definition \ref{defet}]
    A morphism $f\colon A \rightarrow B$ of Dirac rings is called \'etale if it is finitely presented and for any solid commutative square of Dirac rings
    \[
\begin{tikzcd}
	A & C \\
	B & {C/I}
	\arrow[from=1-1, to=1-2]
	\arrow["f"', from=1-1, to=2-1]
	\arrow[from=1-2, to=2-2]
	\arrow[dashed, from=2-1, to=1-2]
	\arrow[from=2-1, to=2-2]
\end{tikzcd},
\] where the right vertical map is the projection onto the quotient by a nilpotent ideal $I \subset C$, there exists a unique dashed arrow making the whole diagram commute.
\end{definition}

\begin{definition}[Definition \ref{defpi*et}]
Let $1 \leq k \leq \infty$.
\begin{enumerate}
    \item A map $\phi\colon A \rightarrow B$ of $\mathbb{E}_{k}$-ring spectra is called $\pi_{*}$-\'etale or a $\pi_{*}$-\'etale extension of $A$ if the induced map $\pi_*(\phi)\colon \pi_*(A) \rightarrow \pi_*(B)$
     of Dirac rings is \'etale.
     \item More generally, given an $\mathbb{E}_{k+1}$-ring spectrum $R$, a map $\phi\colon A \rightarrow B$ of $\mathbb{E}_{k}$-$R$-algebras is called $\pi_{*}$-\'etale if the underlying map of $\mathbb{E}_k$-ring spectra is $\pi_{*}$-\'etale.
\end{enumerate}
      Let $\Alg_{\mathbb{E}_{k}}(\LMod_{R})^{\pi_{*}\text{-\'et}}_{A/} \subset \Alg_{\mathbb{E}_{k}}(\LMod_{R})_{A/}$ denote the full subcategory spanned by $\pi_{*}$-\'etale extensions with source $A \in \Alg_{\mathbb{E}_{k}}(\LMod_{R})$.
\end{definition}

Lurie's \'etale rigidity theorem identifies $\pi_0$-\'etale extensions with their maps on $\pi_0$:

\begin{theorem}[{\cite[Theorem~7.5.0.6]{HA}}]\label{lurieetrig}
    The functor $\pi_0\colon \LMod_{R} \rightarrow \LMod_{\pi_0(R)}(\operatorname{Ab})$ induces an equivalence of categories $\Alg_{\mathbb{E}_{k}}(\LMod_{R})^{\pi_{0}\text{-\'et}}_{A/} \xrightarrow{\simeq}  \CAlg_{\pi_{0}(A)}^{ \text{\'et}}(\operatorname{Ab})$, where the left-hand side denotes the full subcategory of $\Alg_{\mathbb{E}_{k}}(\LMod_{R})_{A/}$ spanned by $\pi_0$-\'etale morphisms and the right-hand side denotes the 1-category of \'etale $\pi_0(A)$-algebras.
\end{theorem}

Our relative version uses the full graded homotopy ring:

\begin{theorem}[Theorem \ref{k+1}]
The functor $\pi_{*}\colon \LMod_{R} \rightarrow \LMod_{\pi_{*}(R)}(\operatorname{grAb})$ induces an equivalence of categories $\Alg_{\mathbb{E}_{k}}(\LMod_{R})^{\pi_{*}\text{-\'et}}_{A/} \xrightarrow{\simeq}  \CAlg_{\pi_{*}(A)}^{\text{\'et}}(\operatorname{grAb})$, where the right-hand side denotes the 1-category of \'etale Dirac $\pi_{*}(A)$-algebras.
\end{theorem}

\begin{remark}
When $A \simeq R$ is the initial object in $\Alg_{\mathbb{E}_{k}}(\LMod_R)$, an $\mathbb{E}_{k}$-algebra $B$ is \'etale in the sense of \cite{DG1} if and only if the unit map $\eta\colon R \rightarrow B$ is $\pi_{*}$-\'etale in our sense. In this case, the theorem above is stated and proved in \cite[Theorem~4.33]{DG1}. Applying \cite[Theorem~4.33]{DG1} with base $A$ requires an $\mathbb E_{k+1}$-structure on $A$. Our hypotheses do not require such a structure on $A$.
\end{remark}

\subsection{Outline of proof}\label{outline}
Lurie's proof of Theorem \ref{lurieetrig} passes to connective covers and proceeds along the Postnikov tower. To retain the full graded homotopy ring, we instead use the potential-stage tower in synthetic modules.

For an $\mathbb E_{k+1}$-ring spectrum $R$, the complete prestable category $\operatorname{Syn}_R$ has graded $\pi_*R$-modules as its heart and recovers $R$-module spectra as its periodic objects \cite[Section~6.5]{patchkoria2023adams}. Thus stage zero of the tower records $\pi_*$ rather than $\pi_0$.

Building on \cite{pstragowski2021abstract}, we construct relative obstruction classes for both objects and morphisms under a fixed algebra $A$.  A base-change argument identifies every relative obstruction space at a finite stage with a mapping space out of the algebraic relative cotangent complex
\[
    \mathbb L^{\mathbb E_k}_{B_0/A_0}.
\]
For $k\geq2$ this complex vanishes for an \'etale Dirac map by \cite[Proposition~4.35]{DG1}; for $k=1$, the required mapping spaces vanish by the augmentation-ideal calculation $I/I^2\simeq\Omega^1_{B_0/A_0}=0$.  Hence every transition in the relative potential-stage tower is an equivalence.  Appendix~\ref{app:dunn-additivity} gives a second proof for $k\geq2$ by Dunn additivity, using the case $k=1$ and following \cite[Proof of Theorem~7.5.4.2]{HA}.

\subsection{Notation and conventions}
 By a category we mean an $\infty$-category; ordinary categories are called $1$-categories. All tensor products and mapping objects are derived unless explicitly stated otherwise. If $\mathcal D$ is a stable $\infty$-category equipped with a $t$-structure, its heart is denoted by $\mathcal D^\heartsuit$. We use the same superscript for internal objects in the heart, for example $\CAlg^\heartsuit$. We use homological grading unless stated otherwise.

\section*{Acknowledgments}
This paper is based on my master's thesis, written under the supervision of Robert Burklund. I thank him for his guidance, support, and many helpful conversations. I also thank Pengkun Huang, Andreas Studsgaard, Peixiang Tan, and Sophus Willumsgaard for useful discussions and encouragement. I am especially grateful to Hongxiang Zhao for carefully reading the first draft and suggesting revisions.

\section{Goerss--Hopkins obstruction theory}
We recall the potential-stage tower of abstract Goerss--Hopkins theory and its description of lift spaces in terms of cotangent complexes \cite{pstragowski2021abstract}.

\subsection{Prestable categories and synthetic modules}
We begin with Grothendieck prestable categories \cite[Appendix~C]{SAG}.

\begin{definition}
    A category $\mathcal{C}$ is called a Grothendieck prestable category if there exists a stable category $\mathcal{D}$ with a $t$-structure $(\mathcal{D}_{\geq 0},\mathcal{D}_{\leq 0})$ compatible with filtered colimits and an equivalence $\mathcal{C} \xrightarrow{\simeq} \mathcal{D}_{\geq 0}$. In particular, $\mathcal{C}$ is additive and presentable.
\end{definition}

\begin{remark}
    Identify a Grothendieck prestable category $\mathcal{C}$ with $\mathcal{D}_{\geq 0}$. Its heart $\mathcal{C}^{\heartsuit}$ is the full subcategory of $\mathcal{C}$ spanned by discrete objects and identifies with $\mathcal{D}^{\heartsuit}$.
\end{remark}

\begin{definition}
    Let $2 \leq k \leq \infty$. A \textit{grading} on an $\mathbb{E}_k$-monoidal Grothendieck prestable category $\mathcal{C}$ is an autoequivalence denoted by $c \mapsto c[1]$, equipped with a natural equivalence $c[1] \otimes d \xrightarrow{\simeq} (c \otimes d)[1]$.
\end{definition}

\begin{remark}
    The grading restricts to a self-equivalence on the heart $\mathcal{C}^{\heartsuit}$, making it a graded abelian category. The grading on the heart is compatible with that on $\mathcal{C}$: there is a canonical equivalence $\pi_{i}(A[k]) \xrightarrow{\simeq} (\pi_{i}(A)) [k]$ in $\mathcal{C}^{\heartsuit}$.
\end{remark}

\begin{warning}
    The grading autoequivalence should be distinguished from suspension, which need not be invertible in a Grothendieck prestable category.
\end{warning}

\begin{definition}
    A Grothendieck prestable category is said to be separated if taking homotopy groups detects equivalences. It is called complete if, in addition, every Postnikov tower converges.
\end{definition}

We will apply the obstruction theory to the category of synthetic $R$-modules defined below. This construction appears in \cite[Section~6.5]{patchkoria2023adams} and is recalled in \cite[Proof of Theorem~4.33]{DG1}.

\begin{construction}\label{SynthR}
    Let $R$ be an $\mathbb{E}_{k+1}$-ring spectrum, where $1\leq k \leq \infty$. Let $\LMod^{\mathrm{ff}}_{R} \subset \LMod_R$ be the full subcategory spanned by finite sums of suspensions of $R$. Write \[
\operatorname{Syn}_R = \Psh^{\Sigma}(\LMod^{\mathrm{ff}}_{R}, \mathcal{S})
\] for the category of presheaves on $\LMod^{\mathrm{ff}}_{R}$ that preserve finite products. Its objects are called synthetic $R$-modules.
\end{construction}

    We use the following properties of $\operatorname{Syn}_R$:
\begin{enumerate}
    \item The category $\operatorname{Syn}_{R}$ is a complete Grothendieck prestable category. Every object of $\operatorname{Syn}_{R}$ therefore identifies with a connective object of $\Sp(\operatorname{Syn}_{R})$.
    \item By Day convolution, $\operatorname{Syn}_{R}$ inherits an $\mathbb{E}_k$-monoidal structure from $\LMod^{\mathrm{ff}}_{R}$. The restricted Yoneda embedding $\nu\colon \LMod_{R} \rightarrow \operatorname{Syn}_{R}$ is $\mathbb{E}_k$-monoidal, and the monoidal unit of $\operatorname{Syn}_{R}$ is $\nu R$.
    \item The natural grading $X[1](T) \xrightarrow{\simeq} X(\Sigma^{-1} T)$ is induced by suspension in $\LMod^{\mathrm{ff}}_{R}$. The heart admits a canonical equivalence $\operatorname{Syn}_{R}^{\heartsuit} \xrightarrow{\simeq} \LMod_{\pi_{*}(R)}(\operatorname{grAb})$ under which $\pi_{0}(\nu A) \simeq \pi_{*}(A)$ as graded $\pi_{*}(R)$-modules.
\end{enumerate}

\subsection{Goerss--Hopkins tower}
We use the abstract Goerss--Hopkins obstruction theory of Pstr\k{a}gowski and VanKoughnett.
Fix a graded symmetric monoidal, separated Grothendieck prestable category $\mathcal{C}$. We omit $\mathcal{C}$ from the notation when the context is clear. For example, for $A \in \Alg(\mathcal{C})$, we write $\LMod_A$ for $\LMod_A(\mathcal{C})$.

\begin{definition}
    A shift algebra $A$ is an associative algebra $A \in \Alg(\mathcal{C})$ equipped with a map $\tau\colon \Sigma A[-1] \rightarrow A$ of right $A$-modules which induces an isomorphism $\pi_*(A) \xrightarrow{\simeq} \pi_{0}(A)[\tau]$ in $\mathcal{C}^{\heartsuit}$, where the latter is the graded algebra in $\mathcal{C}^{\heartsuit}$ given by $(\pi_{0}(A)[\tau])_k \simeq \pi_0(A)[-k]$.
\end{definition}

\begin{definition}
    Let $A$ be a shift algebra and $M$ be a left $A$-module. We say that $M$ is a periodic $A$-module if $\pi_*(M) \simeq \pi_*(A) \otimes_{\pi_0(A)} \pi_0(M)$. Let $\LMod_{A}^{per} \subset \LMod_{A}$ denote the full subcategory spanned by periodic modules.
\end{definition}

\begin{lemma}
    Let $A$ be a shift algebra and $M$ be a left $A$-module. $M$ is a periodic $A$-module if and only if $\pi_0(A) \otimes_A M$ lies in $\mathcal{C}^{\heartsuit}$.
\end{lemma}
\begin{proof}
    See \cite[Proposition~2.16]{pstragowski2021abstract}.
\end{proof}

Compatibility of the symmetric monoidal structure with the $t$-structure makes the Postnikov truncation functor $\tau_{\leq n}$ lax symmetric monoidal. Each associative algebra $A$ therefore has a tower of associative algebras \[
\dots \rightarrow A_n \rightarrow A_{n-1} \rightarrow \dots \rightarrow A_0,
\] where $A_n\simeq\tau_{\leq n}A$. This induces a tower of Grothendieck prestable categories \[
\dots \rightarrow \LMod_{A_{n}} \rightarrow \LMod_{A_{n-1}} \rightarrow \dots \rightarrow \LMod_{A_{0}}.
\] If $\mathcal{C}$ is complete, there are canonical equivalences \[
A \xrightarrow{\simeq} \varprojlim A_{n}
\] and \[
\LMod_{A} \xrightarrow{\simeq} \varprojlim \LMod_{A_{n}}.
\]

We use the following notion of a small extension \cite[Section~7.4.1]{HA}:
\begin{definition}\label{nsmall}
    Let $\mathcal{D}$ be a stable presentable category with compatible $\E_k$-monoidal structure and $t$-structure. Assume that the tensor unit $\mathbf{1} \in \mathcal{D}$ is connective.
    Let $f\colon A \rightarrow B$ be a map of $\E_k$-algebras in $\mathcal{D}$ and let $n \geq 0$. We say $f$ is an $n$-small extension if the following conditions are satisfied:
    \begin{enumerate}
        \item $A \in \mathcal{D}_{\geq 0}$ is connective and $\operatorname{fib}(f) \in \mathcal{D}_{\geq n} \cap \mathcal{D}_{\leq 2n}$.
        \item The multiplication map on the nonunital algebra $\operatorname{fib}(f) \otimes \operatorname{fib}(f) \rightarrow \operatorname{fib}(f)$ is nullhomotopic.
    \end{enumerate}
\end{definition}

Lurie proved that any $n$-small extension $f\colon A' \rightarrow A$ with fiber $M \in \Mod^{\E_k}_A$ is uniquely determined by a derivation $\eta\colon \mathbb{L}_{A}^{\mathbb{E}_k} \rightarrow \Sigma M$ in $\Mod^{\E_k}_A$:

\begin{theorem}[{\cite[Definition~7.4.1.1, Notation~7.4.1.22, and Theorem~7.4.1.26]{HA}}]
    Let $\mathcal{D}$ be as in Definition \ref{nsmall}. Let $\operatorname{Der}^{(k)} = \operatorname{Der}^{(k)}(\Alg_{\E_k}(\mathcal{D}))$ be the category of derivations $(A, \eta\colon \mathbb{L}_{A}^{\mathbb{E}_k} \rightarrow \Sigma M)$ in $\E_k$-algebras. There exists a functor
    \[
\Phi^{(k)}\colon \operatorname{Der}^{(k)} \rightarrow \Fun(\Delta^{1}, \Alg_{\E_k}(\mathcal{D})),
\]
    \[
(A, \eta\colon \mathbb{L}_{A}^{\mathbb{E}_k} \rightarrow \Sigma M) \mapsto (A^{\eta} \rightarrow A).
\] Moreover, let $\operatorname{Der}^{(k)}_{n-sm} \subset \operatorname{Der}^{(k)}$ be the full subcategory spanned by those pairs $(A, \eta\colon \mathbb{L}_{A}^{\mathbb{E}_k} \rightarrow \Sigma M)$ for which $M \in \mathcal{D}_{[n,2n]}$. Then $\Phi^{(k)}$ induces a fully faithful functor
    \[
\Phi^{(k)}\colon \operatorname{Der}^{(k)}_{n-sm} \rightarrow \Fun(\Delta^{1}, \Alg_{\E_k}(\mathcal{D})),
\]
    whose (essential) image is exactly the full subcategory spanned by $n$-small extensions.
\end{theorem}

\begin{corollary}\label{nsmallissquarezero}
    Any $n$-small extension of $\E_k$-algebras is a square-zero extension in the sense that it is given by a derivation.
\end{corollary}

For a shift algebra $A$ and $n \geq 1$, consider the functor $\LMod_{A_{n}} \rightarrow \LMod_{A_{n-1}}$. By \cite{HA}, $A_{n} \rightarrow A_{n-1}$ is a square-zero extension and fits into the cartesian square \[
\begin{tikzcd}
	{A_n} & {A_{n-1}} \\
	{A_{n-1}} & {A_{n-1}\oplus \Sigma F_n}
	\arrow[from=1-1, to=1-2]
	\arrow[from=1-1, to=2-1]
	\arrow["\lrcorner"{anchor=center, pos=0.125}, draw=none, from=1-1, to=2-2]
	\arrow["{d_0}"', from=1-2, to=2-2]
	\arrow["d", from=2-1, to=2-2]
\end{tikzcd}
\] in $\Alg(\mathcal{C})$, where $F_n:=\operatorname{fib}(A_n\to A_{n-1})\simeq\Sigma^n A_0[-n]$ and ${A_{n-1}\oplus \Sigma F_n}$ denotes the trivial square-zero extension of $A_{n-1}$ by $\Sigma F_n$. The maps $d_0$ and $d$ are respectively the trivial derivation and the derivation classifying $A_n \rightarrow A_{n-1}$. The corresponding square of module categories is cartesian:
\[
\begin{tikzcd}
	{\LMod_{A_n}} & {\LMod_{A_{n-1}}} \\
	{\LMod_{A_{n-1}}} & {\LMod_{A_{n-1}\oplus \Sigma F_n}}
	\arrow[from=1-1, to=1-2]
	\arrow[from=1-1, to=2-1]
	\arrow["\lrcorner"{anchor=center, pos=0.125}, draw=none, from=1-1, to=2-2]
	\arrow["{d_0^*}"', from=1-2, to=2-2]
	\arrow["d^*", from=2-1, to=2-2]
\end{tikzcd},
\] where the maps in the second diagram are all extensions of scalars along maps in the first square. The functor $d^*$ admits a right adjoint $d_*$. This gives a lax symmetric monoidal endofunctor $\Theta = d_* d_0^*\colon \LMod_{A_{n-1}} \rightarrow \LMod_{A_{n-1}}$ and a natural transformation $\pi\colon \Theta \simeq d_* d_0^* \Longrightarrow d_* p^* p_* d_0^* \simeq \operatorname{id}$, since both $d$ and $d_0$ are sections of the projection $p\colon A_{n-1} \oplus \Sigma F_n \rightarrow A_{n-1}$. As an object of $\mathcal C$,
\[
    \Theta M\simeq M\oplus(\Sigma F_n\otimes_{A_{n-1}}M)
    \simeq M\oplus\Sigma^{n+1}(A_0\otimes_{A_{n-1}}M)[-n].
\]
Its $A_{n-1}$-module structure depends on $d$.

The category $\LMod_{A_n}$ admits a description by sections of $\pi$.
\begin{construction}
Define $\Theta\text{-}Sect_{A_{n-1}}$ to be the category of triples $(M,s,h)$, where $s$ is a map $M \rightarrow \Theta M$ in $\mathcal{C}$ and $h\colon \pi s \simeq \operatorname{id}$ is a homotopy witnessing that $s$ is a section. The category $\Theta\text{-}Sect$ is defined by the cartesian square
    \[
\begin{tikzcd}
	{\Theta\text{-}Sect_{A_{n-1}}} & {\Fun(\Delta^2 , \LMod_{A_{n-1}})} \\
	{\LMod_{A_{n-1}}} & {\Fun(\Lambda_2^2 , \LMod_{A_{n-1}})}
	\arrow[from=1-1, to=1-2]
	\arrow[from=1-1, to=2-1]
	\arrow["\lrcorner"{anchor=center, pos=0.125}, draw=none, from=1-1, to=2-2]
	\arrow[from=1-2, to=2-2]
	\arrow["{(\pi, \operatorname{id})}", from=2-1, to=2-2]
\end{tikzcd}.
\]
\end{construction}

\begin{proposition}[{\cite[Theorem~3.8]{pstragowski2021abstract}}] \label{thetasect}
    There is a canonical equivalence of categories \[
\LMod_{A_n} \xrightarrow{\simeq} \Theta\text{-}Sect_{A_{n-1}}.
\] Moreover, under this equivalence, the extension of scalars $\LMod_{A_n} \rightarrow \LMod_{A_{n-1}}$ corresponds to the forgetful functor $\Theta\text{-}Sect_{A_{n-1}} \rightarrow \LMod_{A_{n-1}}$ sending $(M,s,h)$ to $M$.
\end{proposition}

\begin{remark}
    By \cite[Remark~3.9]{pstragowski2021abstract}, the equivalence in the proposition admits a symmetric monoidal refinement when $A$ is commutative.
\end{remark}

\begin{example}[{\cite[Proof of Theorem~4.33]{DG1}}]
    Let $\mathcal{C} \simeq \operatorname{Syn}_R$ for an $\mathbb{E}_\infty$-ring spectrum $R$, as in Construction \ref{SynthR}. The monoidal unit $\nu R$ is canonically a shift algebra in $\operatorname{Syn}_R$. The restricted Yoneda embedding induces a canonical equivalence \[
\LMod_R(\Sp) \xrightarrow{\simeq} \LMod_{\nu R}^{per}(\operatorname{Syn}_R) =: \operatorname{Syn}_{R}^{per}
\] between left $R$-module spectra and periodic $\nu R$-modules in $\operatorname{Syn}_R$. We also call the latter periodic $R$-modules or simply periodic objects.
\end{example}

Potential stages extend the notion of periodicity along the Postnikov tower \[
\dots \rightarrow A_n \rightarrow A_{n-1} \rightarrow \dots \rightarrow A_0
\] associated with $A$:

\begin{definition}[{\cite[Definition~4.1]{pstragowski2021abstract}}]
    Assume $\mathcal{C}$ is complete. Let $0 \leq n \leq \infty$. An $A_n$-module $M$ is called a potential $n$-stage for a periodic $A$-module if $A_0 \otimes_{A_n} M$ is discrete. We denote the full subcategory of $\LMod_{A_n}$ spanned by potential $n$-stages by $\mathcal{M}_n$.
\end{definition}

\begin{remark}
    A potential $0$-stage is a discrete $\pi_0(A)$-module, and a potential $\infty$-stage is a periodic $A$-module. The extension-of-scalars functor $\LMod_{A_n} \rightarrow \LMod_{A_m}$ for $m \leq n$ restricts to $\mathcal{M}_n \rightarrow \mathcal{M}_m$. A potential $n$-stage is $n$-truncated, and its $A_n$-module structure is uniquely determined by its underlying $A$-module. The mapping spaces can therefore be computed in $\LMod_A$.
\end{remark}

Fix a shift algebra $A$. The categories of potential $n$-stages form a tower: \[
\dots \rightarrow \mathcal{M}_n \rightarrow \mathcal{M}_{n-1} \rightarrow \dots \rightarrow \mathcal{M}_{0}.
\]
If $\mathcal{C}$ is complete, this tower has inverse limit $\mathcal{M}_{\infty} = \varprojlim\mathcal{M}_{n}$ and an induced map $\mathcal{M}_{\infty} \rightarrow \mathcal{M}_0$.

Typically, $\mathcal M_\infty$ describes the geometric objects to be classified, $\mathcal M_0$ their algebraic invariants, and the functor $\mathcal M_\infty\to\mathcal M_0$ assigns the invariant \cite[Remark~4.3]{pstragowski2021abstract}.

For algebras, we impose the potential-stage condition on the underlying module.
\begin{definition}[{\cite[Definition~5.1]{pstragowski2021abstract}}]
    A potential $n$-stage for an $\mathbb{E}_k$-algebra is an $\mathbb{E}_k$-algebra in $\LMod_{A_n}$ whose underlying $A_n$-module is a potential $n$-stage for a periodic $A$-module. We write $\Alg_{\mathbb E_k}(\mathcal M_n):=\Alg_{\mathbb E_k}(\LMod_{A_n})\times_{\LMod_{A_n}}\mathcal M_n$ for this category.
\end{definition}

\begin{proposition}[{\cite[Proposition~4.7, Corollary~4.8, and Lemma~5.3]{pstragowski2021abstract}}]\label{square-zero2}
    Let $n\geq 1$ and let $R \in \Alg_{\mathbb{E}_k}(\mathcal{M}_{n-1})$ be a potential $(n-1)$-stage for an $\mathbb{E}_k$-algebra. Write $R_0:=\pi_0(R)$. Then $\pi\colon \Theta R \rightarrow R$ is a square-zero extension by $\Sigma^{n+1} R_0[-n]$.
\end{proposition}

\begin{proof}
    The projection $p\colon A_{n-1}\oplus\Sigma^{n+1}A_0[-n]\to A_{n-1}$ is an $n$-equivalence, and so is the map $\pi$ induced by its adjunction unit. Since $R$ is a potential $(n-1)$-stage, $A_0\otimes_{A_{n-1}}R\simeq R_0$, and hence
    \[
        \Theta R\simeq R\oplus\Sigma^{n+1}(A_0\otimes_{A_{n-1}}R)[-n]
        \simeq R\oplus\Sigma^{n+1}R_0[-n]
    \]
    as objects of $\mathcal C$. Since $R$ is $(n-1)$-truncated and $\pi$ is an $n$-equivalence, this identifies $\operatorname{fib}(\pi)$ with $\Sigma^{n+1}R_0[-n]$. The fiber is concentrated in homotopy degree $n+1$, so its multiplication is nullhomotopic by connectivity. Thus $\pi$ is an $(n+1)$-small extension, and Corollary \ref{nsmallissquarezero} shows that it is a square-zero extension.
\end{proof}

The square-zero extension $\pi\colon \Theta R \rightarrow R$ determines the obstruction to lifting a potential $(n-1)$-stage $R$ to a potential $n$-stage:

\begin{proposition}[{\cite[Theorem~5.4]{pstragowski2021abstract}}]\label{obobj}
    Let $R \in \Alg_{\mathbb{E}_k}(\mathcal{M}_{n-1})$ be a potential $(n-1)$-stage for an $\mathbb{E}_k$-algebra, $1 \leq k \leq \infty$. Then there exists an obstruction \[
\theta \in \operatorname{Ext}_{\pi_0(R)}^{n+2}(\mathbb{L}^{\mathbb{E}_k}_{\pi_0(R)}, \pi_0(R)[-n]),
\] which vanishes if and only if $R$ can be lifted to a potential $n$-stage $R' \in \Alg_{\mathbb{E}_k}(\mathcal{M}_{n})$. Here the Ext group is computed in the derived category of $\E_k$-modules in the derived category $\LMod_{\pi_0(A)}$.
\end{proposition}

The proof uses base change for cotangent complexes:

\begin{lemma}[{\cite[Proposition~7.3.3.7]{HA}}]\label{basechange}
    Let $\mathcal C$ be a presentable category with tangent bundle $T_\mathcal C$, and let $p$ be the composite \[
T_\mathcal{C} \rightarrow \Fun(\Delta^1, \mathcal{C}) \rightarrow \Fun(\{1\}, \mathcal{C}) \simeq {C}.
\]
    Suppose that we are given a cocartesian square \[
\begin{tikzcd}
	A & B \\
	{A'} & {B'}
	\arrow[from=1-1, to=1-2]
	\arrow[from=1-1, to=2-1]
	\arrow[from=1-2, to=2-2]
	\arrow[from=2-1, to=2-2]
	\arrow["\lrcorner"{anchor=center, pos=0.125, rotate=180}, draw=none, from=2-2, to=1-1]
\end{tikzcd}
\]
in $\mathcal{C}$. Then the induced map $\beta\colon \mathbb{L}_{B/A} \rightarrow \mathbb{L}_{B'/A'}$ is a $p$-cocartesian morphism in $T_\mathcal{C}$.
\end{lemma}

When $\mathcal{C}$ is $\Alg_{\mathbb{E}_k}(\mathcal{M}_{\infty})$, the lemma shows that the canonical map $g^*\mathbb{L}_{B/A} \rightarrow \mathbb{L}_{B'/A'}$ is an equivalence of $\mathbb{E}_k$-operadic $B'$-modules, where $g^*$ denotes the base change functor $\Mod^{\mathbb{E}_k}_B \simeq \Sp(\mathcal{C}^{/B}) \rightarrow \Sp(\mathcal{C}^{/B'}) \simeq \Mod^{\mathbb{E}_k}_{B'}$.

\begin{proof}[Proof of Proposition \ref{obobj}]
    By Proposition \ref{thetasect}, a potential $(n-1)$-stage $R$ can be lifted to a potential $n$-stage if and only if $\pi\colon \Theta R \rightarrow R$ admits a section. By Proposition \ref{square-zero2}, $\pi$ exhibits $\Theta R$ as a square-zero extension of $R$ by $\Sigma^{n+1} \pi_0(R)[-n]$, and so is classified by an element in $\pi_0 \Map_{R}^{\E_k}(\mathbb{L}^{\mathbb{E}_k}_{R}, \Sigma^{n+2}\pi_0(R)[-n])$, according to \cite[Section~7.4]{HA}. The target of the mapping space is a suspension of a discrete object, and so it is canonically an $A_0$-module. Since $R$ is a potential $(n-1)$-stage, $A_0 \otimes_{A_{n-1}} R \simeq R_0$. Base change therefore gives \[
\pi_0 \Map_{R}^{\mathbb{E}_k}(\mathbb{L}^{\mathbb{E}_k}_{R}, \Sigma^{n+2}\pi_0(R)[-n]) \simeq \pi_0 \Map_{\pi_0(R)}^{\mathbb{E}_k}(\mathbb{L}^{\mathbb{E}_k}_{\pi_0(R)}, \Sigma^{n+2}\pi_0(R)[-n]),
\] where the right-hand side is canonically equivalent to the Ext group in the proposition.
\end{proof}

There are corresponding obstructions to lifting morphisms.

\begin{proposition}[{\cite[Proposition~5.7]{pstragowski2021abstract}}]\label{obmor}
    Let $R, S \in \Alg_{\mathbb{E}_k}(\mathcal{M}_{n})$ be objects whose images in $\Alg_{\mathbb{E}_k}(\mathcal{M}_{n-1})$ are denoted by $R',S'$. Then, given any map $\phi\colon R' \rightarrow S'$, there exists an equivalence
    \begin{equation*}
        \begin{split}
            \mathcal{L}(\phi) &:= \Map_{\Alg_{\E_k}(\mathcal{M}_n)}(R,S) \times_{\Map_{\Alg_{\E_k}(\mathcal{M}_{n-1})}(R',S')} \{ \phi \} \\
            &\xrightarrow{\simeq} P_{0,\eta}\Map_{\pi_0(R)}^{\E_k}(\mathbb{L}_{\pi_0(R)}^{\mathbb{E}_k}, \Sigma^{n+1}\pi_0(S)[-n])
        \end{split}
    \end{equation*}
    between the space of lifts of $\phi$ and the space of nullhomotopies of a certain map $\eta\colon \mathbb{L}_{\pi_0(R)}^{\mathbb{E}_k} \rightarrow \Sigma^{n+1} \pi_0(S)[-n]$ of $\E_k$-$\pi_0(R)$-modules.
\end{proposition}

\begin{remark}
    The equivalence is natural in the following sense. Let $T \in \Alg_{\mathbb{E}_k}(\mathcal{M}_n)$ have image $T' \in \Alg_{\mathbb{E}_k}(\mathcal{M}_{n-1})$, and let $\psi\colon S' \rightarrow T'$ be a map. Then the following diagram commutes: \[
\begin{tikzcd}
	{\mathcal{L}(\phi)} & {P_{0,\eta}\Map_{\pi_0(R)}(\mathbb{L}_{\pi_0(R)}^{\mathbb{E}_k}, \Sigma^{n+1}\pi_0(S)[-n])} \\
	{\mathcal{L}(\psi\circ\phi)} & {P_{0,\eta '}\Map_{\pi_0(R)}(\mathbb{L}_{\pi_0(R)}^{\mathbb{E}_k}, \Sigma^{n+1}\pi_0(T)[-n])}
	\arrow["\simeq", from=1-1, to=1-2]
	\arrow["{\psi \circ -}"', from=1-1, to=2-1]
	\arrow["{\psi^*}", from=1-2, to=2-2]
	\arrow["\simeq", from=2-1, to=2-2]
\end{tikzcd}.
\]
\end{remark}

\section{\texorpdfstring{$\pi_*$-\'etale}{pi-star-etale} morphisms of ring spectra}
From now on, $R$ denotes the base ring, and $A,B,C$ denote its algebras.
\subsection{Relative \texorpdfstring{$\pi_*$-\'etale}{pi-star-etale} rigidity}

We recall the definition of an \'etale morphism of Dirac rings from \cite[Section~4]{DG1}.

\begin{definition} \label{defet}
    A morphism $f\colon A \rightarrow B$ of Dirac rings is called \'etale if it is finitely presented and for any solid commutative square of Dirac rings
    \[
\begin{tikzcd}
	A & C \\
	B & {C/I}
	\arrow[from=1-1, to=1-2]
	\arrow["f"', from=1-1, to=2-1]
	\arrow[from=1-2, to=2-2]
	\arrow[dashed, from=2-1, to=1-2]
	\arrow[from=2-1, to=2-2]
\end{tikzcd},
\] where the right vertical map is the projection onto the quotient by a nilpotent graded ideal $I \subset C$, there exists a unique dashed arrow making the whole diagram commute.
\end{definition}

\'Etale morphisms of Dirac rings are even and flat; moreover, they have vanishing relative cotangent complexes:

\begin{proposition} \label{discretecotang}
   Let $R$ be a Dirac ring and let $f\colon A \rightarrow B$ be an \'etale morphism of Dirac $R$-algebras. We regard $A$ and $B$ as $\mathbb{E}_k$-algebra objects in $\mathcal{D}(R)$, the derived category of graded $R$-modules, for $2 \leq k \leq \infty$. Then the $\mathbb{E}_k$-cotangent complex $\mathbb{L}^{\mathbb{E}_k}_{B/A} \in \Mod_B^{\E_k}(\mathcal{D}(R))$ vanishes.
\end{proposition}

\begin{proof}
    Since $A$ is a commutative algebra object of $\mathcal{D}(R)$, the usual change-of-base equivalence identifies $\mathbb{E}_k$-$A$-algebras in $\mathcal{D}(A)$ with $\mathbb{E}_k$-$R$-algebras under $A$ in $\mathcal{D}(R)$. Under this identification, square-zero extensions of $B$ and derivations out of $B$ agree on the two sides. Hence the relative cotangent complex computed in $\mathcal{D}(R)$ is the image of the relative cotangent complex computed in $\mathcal{D}(A)$.
    Since $f$ is \'etale, \cite[Proposition~4.35]{DG1} gives $\mathbb{L}^{\mathbb{E}_k}_{B/A} \simeq 0$ in $\Mod_B^{\E_k}(\mathcal{D}(A))$. Therefore its image in $\Mod_B^{\E_k}(\mathcal{D}(R))$ is also zero.
\end{proof}

\begin{lemma}\label{discretebimodvanishing}
    Let $f_0\colon A_0\rightarrow B_0$ be an \'etale morphism of Dirac rings. Put
    \[
        B_0^e=B_0\otimes_{A_0}B_0^{\op}
    \]
    and let $\mu\colon B_0^e\rightarrow B_0$ be multiplication. If $N$ is a graded $B_0$-$B_0$-bimodule whose bimodule structure factors through $\mu$, then
    \[
        \Map_{_{B_0}\BMod_{B_0}}\!\left(\mathbb{L}^{\E_1}_{B_0/A_0},N\right)
    \]
    is contractible. Equivalently, all corresponding Ext groups vanish.
\end{lemma}

\begin{proof}
    Let $I=\ker(\mu)$. By \cite[Theorem~7.3.5.1]{HA}, the associative relative cotangent complex $\mathbb{L}^{\E_1}_{B_0/A_0}$ is represented by $I$ as a left module over $B_0^e$, equivalently as a $B_0$-$B_0$-bimodule. Since $f_0$ is \'etale, $B_0$ is flat over $A_0$ and the diagonal
    \[
        \mu\colon B_0\otimes_{A_0}B_0\longrightarrow B_0
    \]
    is flat. Restriction of scalars along $\mu$ has left adjoint
    \[
        B_0\otimes^{\mathbb L}_{B_0^e}(-)\colon D(B_0^e)\longrightarrow D(B_0).
    \]
    If $N$ factors through $\mu$, then $N\simeq \mu_*\overline N$ for a graded $B_0$-module $\overline N$. The adjunction gives
    \[
        \Map_{B_0^e}(I,N)
        \simeq
        \Map_{B_0}\!\left(B_0\otimes^{\mathbb L}_{B_0^e}I,\overline N\right).
    \]
    The flatness of the diagonal removes higher Tor terms, and therefore
    \[
        B_0\otimes^{\mathbb L}_{B_0^e}I
        \simeq
        B_0\otimes_{B_0^e}I
        \simeq
        I/I^2.
    \]
    Finally $I/I^2\simeq\Omega^1_{B_0/A_0}$, and this module vanishes for \'etale morphisms of Dirac rings by \cite[Proposition~4.24]{DG1}. Hence the displayed mapping space is contractible. Suspensions and grading shifts of $N$ do not affect the argument.
\end{proof}

\begin{definition}\label{defpi*et}
Let $1 \leq k \leq \infty$.
\begin{enumerate}
    \item A map $\phi\colon A \rightarrow B$ of $\mathbb{E}_{k}$-ring spectra is called $\pi_{*}$-\'etale or a $\pi_{*}$-\'etale extension of $A$ if the induced map $\pi_*(\phi)\colon \pi_*(A) \rightarrow \pi_*(B)$
     of Dirac rings is \'etale (see Definition \ref{defet}).
     \item More generally, given an $\mathbb{E}_{k+1}$-ring spectrum $R$, a map $\phi\colon A \rightarrow B$ of $\mathbb{E}_{k}$-$R$-algebras is called $\pi_{*}$-\'etale if the underlying map of $\mathbb{E}_k$-ring spectra is $\pi_{*}$-\'etale.
\end{enumerate}

\end{definition}

\begin{remark}
    Here, for $k = 1$, the condition that the homotopy groups form a Dirac ring is part of the definition.
\end{remark}

\begin{definition}
    We call an $\E_1$-ring spectrum $\pi_*$-Dirac if its homotopy groups form a Dirac ring. Clearly, (the underlying $\E_1$-ring spectrum of) an $\E_k$-ring spectrum is $\pi_*$-Dirac for $k \geq 2$.
\end{definition}

      Given $\psi\colon A \rightarrow C$, a map of $\pi_*$-Dirac $\E_k$-$R$-algebras, let
      \[
        \Alg_{\mathbb{E}_{k}}(\LMod_{R})^{\pi_{*}\text{-\'et}}_{A//C}
        \subset
        \Alg_{\mathbb{E}_{k}}(\LMod_{R})_{A//C}
      \]
      denote the full subcategory spanned by those maps $\phi\colon A \rightarrow B$ of $\E_k$-$R$-algebras over $C$ which are $\pi_{*}$-\'etale extensions. In other words, an object of
      $\Alg_{\mathbb{E}_{k}}(\LMod_{R})^{\pi_{*}\text{-\'et}}_{A//C}$ is a commutative diagram
      \[
\begin{tikzcd}
	& B \\
	A && C
	\arrow[from=1-2, to=2-3]
	\arrow["\phi", from=2-1, to=1-2]
	\arrow["\psi", from=2-1, to=2-3]
\end{tikzcd}
\]
      of $\pi_*$-Dirac $\E_k$-$R$-algebras, where $\phi$ is $\pi_*$-\'etale.
      When $C\simeq0$ is terminal, we write \[\Alg_{\mathbb{E}_{k}}(\LMod_{R})^{\pi_{*}\text{-\'et}}_{A/} \simeq \Alg_{\mathbb{E}_{k}}(\LMod_{R})^{\pi_{*}\text{-\'et}}_{A//0}.\]

      On the algebraic side, for a map $A_*\rightarrow C_*$ of Dirac rings, we write
      \[
      \left(\CAlg_{A_*}^{\heartsuit}\right)^{\text{\'et}}_{/C_*}
      :=
      \CAlg_{A_*}^{\heartsuit,\text{\'et}}
      \times_{\CAlg_{A_*}^{\heartsuit}}
      \left(\CAlg_{A_*}^{\heartsuit}\right)_{/C_*}.
      \]
      Thus an object is a triangle
      \[
      \begin{tikzcd}
          & B_* \ar[dr] \\
          A_* \ar[ur] \ar[rr] && C_*
      \end{tikzcd}
      \]
      of Dirac rings in which $A_*\rightarrow B_*$ is \'etale.

\begin{example}[Finite Galois extensions; \cite{baker2005realizability}]
    Let $R$ be an $\mathbb E_\infty$-ring spectrum, let $G$ be a finite group, and let
    \[
        S\in \CAlg(\LMod_R)^{BG}
    \]
    be a commutative $R$-algebra equipped with a $G$-action. Suppose that $|G|$ is invertible in $\pi_0R$, that $\pi_*S$ is a projective graded $\pi_*R$-module, and that $S$ is $G$-Galois over $R$ in the following sense:
    \begin{enumerate}
        \item the canonical map $R\to S^{hG}$ is an equivalence in $\CAlg(\LMod_R)$;
        \item the descent map
        \[
            S\otimes_R S \longrightarrow \prod_{g\in G} S,
            \qquad
            x\otimes y\longmapsto (x\cdot g(y))_{g\in G},
        \]
        is an equivalence of commutative $S$-algebras.
    \end{enumerate}
    Then the unit $R\to S$ is $\pi_*$-\'etale.
\end{example}

\begin{example}
    Let $p$ be an odd prime, and let
    \[
        \eta\colon L \longrightarrow KU_{(p)}
    \]
    denote the inclusion of the $p$-local periodic Adams summand $L=E(1)$ into $p$-local complex $K$-theory. Passing to $p$-completions, the Adams operations give an action of $C_{p-1}$ on $KU_p$ over $L_p$, and this exhibits $KU_p$ as a $C_{p-1}$-Galois extension, hence a $\pi_*$-\'etale extension of $L_p$. By a straightforward computation, identifying $v_1$ with $\beta^{p-1}$, the induced morphism on homotopy groups is
    \[
        \Z_p[v_1^{\pm 1}] \simeq \Z_p[\beta^{\pm (p-1)}]
        \longrightarrow
        \Z_p[\beta^{\pm 1}].
    \]
    This map is not $\pi_0$-\'etale in the sense of \cite{HA}: it induces an isomorphism on $\pi_0$, but $\pi_2L_p=0$ and $\pi_2KU_p\simeq\Z_p$.
\end{example}

\begin{theorem}[$\pi_*$-\'etale rigidity] \label{k+1}
    Let $R$ be an $\mathbb{E}_{k+1}$-ring spectrum, and let $A \rightarrow C$ be a map of $\pi_*$-Dirac $\mathbb{E}_k$-$R$-algebras, where $1\leq k \leq \infty$. The functor $\pi_{*}\colon \Mod_{R} \simeq \Mod_{R}(\Sp) \rightarrow \Mod_{\pi_{*}(R)}(\operatorname{grAb})$ induces an equivalence of categories
    \[
\Alg_{\mathbb{E}_{k}}(\Mod_{R})^{\pi_{*}\text{-\'et}}_{A//C} \xrightarrow{\simeq}  \left(\CAlg_{\pi_{*}(A)}^{\heartsuit}\right)^{\text{\'et}}_{/\pi_*(C)},
\]
    where the right-hand side is the relative category defined above: its objects are triangles $\pi_*(A)\to B_*\to\pi_*(C)$ in which $\pi_*(A)\to B_*$ is an \'etale Dirac map.
\end{theorem}

Before proving the theorem, we show that $\pi_*$-\'etale extensions inherit higher algebra structures.

\begin{corollary}\label{autoEk}
    Let $A$ be an $\mathbb{E}_{k+1}$-$R$-algebra for an $\mathbb{E}_{k+2}$-ring spectrum $R$, $1 \leq k \leq \infty$. Let $\eta\colon A \rightarrow B$ be a $\pi_*$-\'etale extension of $\mathbb{E}_1$-ring spectra. Then $\eta$ exhibits $B$ as a canonical $\mathbb{E}_k$-$A$-algebra. In particular, $\eta$ is canonically a map of $\mathbb{E}_{k+1}$-$R$-algebras.
\end{corollary}

\begin{proof}
    Combining Theorem~\ref{k+1} and~\cite[Theorem~4.33]{DG1}, we have the canonical equivalences
    \[
        \Alg_{\mathbb{E}_{1}}(\Sp)^{\pi_{*}\text{-\'et}}_{A/}
        \xrightarrow{\simeq}
        \CAlg_{\pi_{*}(A)}^{\heartsuit,\text{\'et}}
        \xleftarrow{\simeq}
        \Alg_{\mathbb{E}_{k}}(\Mod_{A})^{\pi_{*}\text{-\'et}}.
    \]
    Here the last category denotes the category of
    $\mathbb{E}_k$-$A$-algebras which are $\pi_*$-\'etale over $A$.
\end{proof}

We now prove Theorem \ref{k+1} directly by relative obstruction theory in
synthetic $R$-modules.  Let $R$ be an $\mathbb E_{k+1}$-ring spectrum,
$1\leq k\leq\infty$.  The category $\operatorname{Syn}_R$ is complete and
$\mathbb E_k$-monoidal, its heart is canonically equivalent to the category of
graded $\pi_*R$-modules, and periodic objects identify $\mathbb E_k$-monoidally
with $R$-modules in spectra.  Although \cite[Section~5]{pstragowski2021abstract}
is formulated for a symmetric monoidal prestable category, the construction of
the tower for $\mathbb E_k$-algebras only uses the ambient $\mathbb E_k$-monoidal
structure; this is recorded explicitly in \cite[Proof of Theorem~4.33]{DG1}.

Fix $A\in\Alg_{\mathbb E_k}(\Mod_R)$ and write $A_n$ for its image in the
category of potential $n$-stages.  We first formulate the relative forms of the
object and morphism obstruction theories for arbitrary $k$.  When $k=1$, the
category of $\mathbb E_1$-operadic modules is the category of bimodules.

\begin{lemma}
\label{relmorob}
    Let $n\geq 1$, let $A_n,B_n,D_n\in\Alg_{\mathbb E_k}(\mathcal M_n)$,
    and suppose that maps $A_n\to B_n$ and $A_n\to D_n$ have been fixed.
    Let
    \[
        g_{n-1}\colon B_{n-1}\longrightarrow D_{n-1}
    \]
    be a map under $A_{n-1}$.  Then there is a relative obstruction class
    \[
        \eta_{B/A}(g_{n-1})\in
        \Map_{\Mod^{\mathbb E_k}_{B_{n-1}}}\!\left(
            \mathbb L^{\mathbb E_k}_{B_{n-1}/A_{n-1}},
            \Sigma^{n+1}D_0[-n]
        \right)
    \]
    such that the space of lifts of $g_{n-1}$ to a map $B_n\to D_n$ under
    $A_n$ is canonically equivalent to the path space
    \[
        P_{0,\eta_{B/A}(g_{n-1})}
        \Map_{\Mod^{\mathbb E_k}_{B_{n-1}}}\!\left(
            \mathbb L^{\mathbb E_k}_{B_{n-1}/A_{n-1}},
            \Sigma^{n+1}D_0[-n]
        \right).
    \]
\end{lemma}

\begin{proof}
    Put $M=\Sigma^{n+1}D_0[-n]$ and let
    \[
        X_B=\Map_{\Mod^{\mathbb E_k}_{B_{n-1}}}
        (\mathbb L^{\mathbb E_k}_{B_{n-1}},M)
    \]
    and
    \[
        X_A=\Map_{\Mod^{\mathbb E_k}_{B_{n-1}}}\!\left(
        B_{n-1}\boxtimes_{A_{n-1}}
        \mathbb L^{\mathbb E_k}_{A_{n-1}},M\right).
    \]
    By extension--restriction adjunction, $X_A$ is the mapping space that
    occurs when Proposition \ref{obmor} is applied to the restriction of
    $g_{n-1}$ along $A_{n-1}\to B_{n-1}$.  Proposition \ref{obmor} gives an
    absolute obstruction $\eta_B\in X_B$ and an obstruction
    $\eta_A\in X_A$ for the restricted map.  Functoriality of square-zero
    extensions and derivations identifies the image of $\eta_B$ in $X_A$
    with $\eta_A$.

    The given map $A_n\to D_n$ determines a chosen nullhomotopy of $\eta_A$.
    The transitivity cofiber sequence
    \[
        B_{n-1}\boxtimes_{A_{n-1}}
        \mathbb L^{\mathbb E_k}_{A_{n-1}}
        \longrightarrow
        \mathbb L^{\mathbb E_k}_{B_{n-1}}
        \longrightarrow
        \mathbb L^{\mathbb E_k}_{B_{n-1}/A_{n-1}}
    \]
    therefore identifies the pair consisting of $\eta_B$ and this chosen
    nullhomotopy with a point
    \[
        \eta_{B/A}(g_{n-1})\in
        \Map_{\Mod^{\mathbb E_k}_{B_{n-1}}}
        (\mathbb L^{\mathbb E_k}_{B_{n-1}/A_{n-1}},M).
    \]
    Taking the fiber of the absolute path-space description in Proposition
    \ref{obmor} over the chosen nullhomotopy gives the asserted relative
    path-space description.
\end{proof}

\begin{lemma}
\label{relobjob}
    Let $n\geq 1$, let $A_n$ be a fixed lift of $A_{n-1}$, and let
    $f_{n-1}\colon A_{n-1}\to B_{n-1}$ be an object of
    $\Alg_{\mathbb E_k}(\mathcal M_{n-1})_{A_{n-1}/}$.  There is a relative
    obstruction class
    \[
        \theta_{B/A}\in
        \Map_{\Mod^{\mathbb E_k}_{B_{n-1}}}\!\left(
            \mathbb L^{\mathbb E_k}_{B_{n-1}/A_{n-1}},
            \Sigma^{n+2}B_0[-n]
        \right)
    \]
    whose space of nullhomotopies is canonically equivalent to the space of
    lifts of $f_{n-1}$ to an object $A_n\to B_n$ of
    $\Alg_{\mathbb E_k}(\mathcal M_n)_{A_n/}$.
\end{lemma}

\begin{proof}
    Put $M=\Sigma^{n+2}B_0[-n]$.  By Proposition \ref{thetasect}, a lift of
    $B_{n-1}$ to a potential $n$-stage is the same as a section of
    $\Theta B_{n-1}\to B_{n-1}$.  Proposition \ref{square-zero2} and
    square-zero deformation theory identify the space of such sections with
    the nullhomotopy space of an absolute class
    \[
        \theta_B\in
        X_B:=\Map_{\Mod^{\mathbb E_k}_{B_{n-1}}}
        (\mathbb L^{\mathbb E_k}_{B_{n-1}},M).
    \]
    The fixed lift $A_n$ supplies a chosen nullhomotopy of the restriction of
    $\theta_B$ in
    \[
        X_A:=\Map_{\Mod^{\mathbb E_k}_{B_{n-1}}}\!\left(
        B_{n-1}\boxtimes_{A_{n-1}}
        \mathbb L^{\mathbb E_k}_{A_{n-1}},M\right).
    \]
    Applying $\Map(-,M)$ to the transitivity cofiber sequence identifies this
    compatible pair with the class $\theta_{B/A}$ displayed above.  Its
    nullhomotopies are exactly the sections compatible with the fixed section
    over $A_{n-1}$, hence exactly the required relative lifts.
\end{proof}

The obstruction coefficients are operadic $B_0$-modules, so base change and extension--restriction adjunction reduce the obstruction spaces to stage zero.

\begin{lemma}
\label{relativebasechange}
    Let $n\geq1$ and let
    \[
        f_{n-1}\colon A_{n-1}\longrightarrow B_{n-1}
    \]
    be a morphism of potential $(n-1)$-stages for
    $\mathbb E_k$-algebras. If $N$ is an operadic $B_0$-module in the
    derived category of graded $\pi_*R$-modules, regarded by restriction as
    a $B_{n-1}$-module, then there is a canonical equivalence
    \[
    \begin{split}
        &\Map_{\Mod^{\mathbb E_k}_{B_{n-1}}}\!\left(
            \mathbb L^{\mathbb E_k}_{B_{n-1}/A_{n-1}},
            N
        \right)
        \\
        &\qquad\simeq
        \Map_{\Mod^{\mathbb E_k}_{B_0}}\!\left(
            \mathbb L^{\mathbb E_k}_{B_0/A_0},
            N
        \right).
    \end{split}
    \]
\end{lemma}

\begin{proof}
    Let $\mathbf R$ denote the shift algebra in $\operatorname{Syn}_R$, and
    consider the stage-zero extension-of-scalars functor
    \[
        \Phi_n
        :=
        \mathbf R_0\boxtimes_{\mathbf R_{n-1}}(-)\colon
        \LMod_{\mathbf R_{n-1}}(\operatorname{Syn}_R)
        \longrightarrow
        \LMod_{\mathbf R_0}(\operatorname{Syn}_R).
    \]
    By the defining property of potential $(n-1)$-stages, there are
    canonical equivalences
    \[
        \Phi_n(A_{n-1})\simeq A_0,
        \qquad
        \Phi_n(B_{n-1})\simeq B_0.
    \]

    For $X=A_{n-1}$ or $X=B_{n-1}$, let
    \[
        \Phi_{n,X}\colon
        \Mod^{\mathbb E_k}_{X}
        \longrightarrow
        \Mod^{\mathbb E_k}_{X_0}
    \]
    denote the induced extension-of-scalars functor on operadic module
    categories. The base-change compatibility of cotangent complexes under
    cocontinuous monoidal functors gives canonical equivalences
    \[
        \Phi_{n,X}\!\left(
            \mathbb L^{\mathbb E_k}_{X/\mathbf R_{n-1}}
        \right)
        \simeq
        \mathbb L^{\mathbb E_k}_{X_0/\mathbf R_0}.
    \]
    For the $\mathbb E_\infty$ case, this is
    \cite[Lemma~7.6]{pstragowski2021abstract}; the same proof applies to
    $\mathbb E_k$-cotangent complexes, as noted in
    \cite[Proof of Proposition~4.35]{DG1}. This is also the stage-zero
    base-change argument used in
    \cite[Proof of Theorem~5.4]{pstragowski2021abstract}.

    Apply $\Phi_{n,B}$ to the transitivity cofiber sequence
    \[
        B_{n-1}\boxtimes_{A_{n-1}}
        \mathbb L^{\mathbb E_k}_{A_{n-1}/\mathbf R_{n-1}}
        \longrightarrow
        \mathbb L^{\mathbb E_k}_{B_{n-1}/\mathbf R_{n-1}}
        \longrightarrow
        \mathbb L^{\mathbb E_k}_{B_{n-1}/A_{n-1}}.
    \]
    Since $\Phi_{n,B}$ is a colimit-preserving functor between stable
    operadic module categories, it preserves this cofiber sequence.
    Moreover, extension of scalars is compatible with $f_{n-1}$, so the
    first two terms become
    \[
        B_0\boxtimes_{A_0}
        \mathbb L^{\mathbb E_k}_{A_0/\mathbf R_0}
        \longrightarrow
        \mathbb L^{\mathbb E_k}_{B_0/\mathbf R_0}.
    \]
    By naturality, this is the transitivity map associated with
    \[
        \mathbf R_0\longrightarrow A_0\longrightarrow B_0.
    \]
    Its cofiber is therefore
    $\mathbb L^{\mathbb E_k}_{B_0/A_0}$. Consequently, there is a canonical
    equivalence
    \[
        \Phi_{n,B}\!\left(
            \mathbb L^{\mathbb E_k}_{B_{n-1}/A_{n-1}}
        \right)
        \simeq
        \mathbb L^{\mathbb E_k}_{B_0/A_0}.
    \]

    Finally, since $N$ is already an operadic $B_0$-module,
    extension--restriction adjunction gives
    \[
    \begin{split}
        &\Map_{\Mod^{\mathbb E_k}_{B_{n-1}}}\!\left(
            \mathbb L^{\mathbb E_k}_{B_{n-1}/A_{n-1}},
            N
        \right)
        \\
        &\qquad\simeq
        \Map_{\Mod^{\mathbb E_k}_{B_0}}\!\left(
            \Phi_{n,B}\!\left(
                \mathbb L^{\mathbb E_k}_{B_{n-1}/A_{n-1}}
            \right),
            N
        \right)
        \\
        &\qquad\simeq
        \Map_{\Mod^{\mathbb E_k}_{B_0}}\!\left(
            \mathbb L^{\mathbb E_k}_{B_0/A_0},
            N
        \right),
    \end{split}
    \]
    as required.
\end{proof}

\begin{corollary}
\label{relativevanishing}
    Suppose that $A_0\to B_0$ is an \'etale map of Dirac rings.  In the
    situation of Lemma \ref{relativebasechange}, the displayed mapping space
    is contractible in either of the following cases:
    \begin{enumerate}
        \item $2\leq k\leq\infty$ and $N$ is any operadic $B_0$-module;
        \item $k=1$ and the $B_0$-$B_0$-bimodule structure on $N$ factors
        through the multiplication
        $B_0\otimes_{A_0}B_0^{\op}\to B_0$.
    \end{enumerate}
\end{corollary}

\begin{proof}
    By Lemma \ref{relativebasechange}, it suffices to work at stage zero.  For
    $k\geq2$, Proposition \ref{discretecotang} gives
    $\mathbb L^{\mathbb E_k}_{B_0/A_0}\simeq0$.  For $k=1$, the assertion is
    Lemma \ref{discretebimodvanishing}.
\end{proof}

For every $0\leq n\leq\infty$, define
\[
    \mathcal E_n(A)
    :=
    \Alg_{\mathbb E_k}(\mathcal M_n)_{A_n/}
    \times_{\Alg_{\mathbb E_k}(\mathcal M_0)_{A_0/}}
    \CAlg_{A_0}^{\heartsuit,\text{\'et}}.
\]
Thus $\mathcal E_n(A)$ is the full subcategory of the under-category whose
stage-zero map is an \'etale morphism of Dirac rings.

\begin{proposition}
\label{stagewiserigidity}
    For every $n\geq1$, the transition functor
    \[
        \mathcal E_n(A)\longrightarrow\mathcal E_{n-1}(A)
    \]
    is an equivalence of $\infty$-categories.
\end{proposition}

\begin{proof}
    Let $A_{n-1}\to B_{n-1}$ be an object of $\mathcal E_{n-1}(A)$.  By
    Lemma \ref{relobjob}, its relative object obstruction is valued in
    \[
        \Map_{\Mod^{\mathbb E_k}_{B_{n-1}}}\!\left(
            \mathbb L^{\mathbb E_k}_{B_{n-1}/A_{n-1}},
            \Sigma^{n+2}B_0[-n]
        \right).
    \]
    This space is contractible by Corollary \ref{relativevanishing}.  For
    $k=1$, the coefficient bimodule is $B_0$ itself and therefore factors
    through multiplication.  Hence a lift exists and its space of choices is
    contractible.  This proves essential surjectivity.

    To prove full faithfulness, let $B_n,D_n\in\mathcal E_n(A)$ and let
    $g_{n-1}\colon B_{n-1}\to D_{n-1}$ be a map under $A_{n-1}$.  Lemma
    \ref{relmorob} identifies the space of lifts of $g_{n-1}$ with a path
    space in
    \[
        \Map_{\Mod^{\mathbb E_k}_{B_{n-1}}}\!\left(
            \mathbb L^{\mathbb E_k}_{B_{n-1}/A_{n-1}},
            \Sigma^{n+1}D_0[-n]
        \right).
    \]
    This mapping space is again contractible.  If $k=1$, the map of Dirac
    rings $B_0\to D_0$ makes the two $B_0$-actions on $D_0$ agree, so the
    coefficient bimodule factors through multiplication.  Thus every
    lower-stage map has a contractibly unique lift, which proves full
    faithfulness.
\end{proof}

\begin{theorem}
\label{underrigidity}
    Let $R$ be an $\mathbb E_{k+1}$-ring spectrum and let $A$ be a
    $\pi_*$-Dirac $\mathbb E_k$-$R$-algebra.  Taking homotopy groups induces
    an equivalence
    \[
        \Alg_{\mathbb E_k}(\Mod_R)^{\pi_*\text{-\'et}}_{A/}
        \xrightarrow{\simeq}
        \CAlg_{\pi_*A}^{\heartsuit,\text{\'et}}.
    \]
\end{theorem}

\begin{proof}
    Completeness of $\operatorname{Syn}_R$ and the convergence of the
    potential-stage tower give
    \[
        \mathcal E_\infty(A)\simeq\varprojlim_n\mathcal E_n(A).
    \]
    Proposition \ref{stagewiserigidity} identifies this limit with
    $\mathcal E_0(A)=\CAlg_{A_0}^{\heartsuit,\text{\'et}}$.  Under the
    periodic-object equivalence
    $\operatorname{Syn}_R^{\mathrm{per}}\simeq\Mod_R$, the functor from the
    top to the bottom of the tower is precisely $\pi_*$.  This proves the
    assertion.
\end{proof}

We now impose a fixed target.  The following mapping-space statement is the
relative input required to pass from the under-category to the category of
triangles.

\begin{proposition}
\label{ffness}
    Let $A\to B$ be a $\pi_*$-\'etale morphism of $\pi_*$-Dirac
    $\mathbb E_k$-$R$-algebras, and let $A\to C$ be any morphism of
    $\pi_*$-Dirac $\mathbb E_k$-$R$-algebras.  Then taking homotopy groups
    induces an equivalence
    \[
        \Map_{\Alg_{\mathbb E_k}(\Mod_R)_{A/}}(B,C)
        \xrightarrow{\simeq}
        \Map_{\CAlg^{\heartsuit}_{\pi_*A}}(\pi_*B,\pi_*C).
    \]
    The mapping space on the right is discrete.
\end{proposition}

\begin{proof}
    Write $A_n,B_n,C_n$ for the associated potential stages.  Given a map
    $g_{n-1}\colon B_{n-1}\to C_{n-1}$ under $A_{n-1}$, Lemma \ref{relmorob}
    identifies its lift space with a path space in
    \[
        \Map_{\Mod^{\mathbb E_k}_{B_{n-1}}}\!\left(
            \mathbb L^{\mathbb E_k}_{B_{n-1}/A_{n-1}},
            \Sigma^{n+1}C_0[-n]
        \right).
    \]
    Corollary \ref{relativevanishing} makes this space contractible.  In the
    associative case, the coefficient bimodule factors through multiplication
    because $B_0\to C_0$ is a map of Dirac rings.  Thus every transition in
    the tower of mapping spaces is an equivalence.  At stage zero the mapping
    space is the discrete set of maps of Dirac $A_0$-algebras.  Passing to the
    inverse limit proves the claim.
\end{proof}

\begin{proof}[Proof of Theorem \ref{k+1}]
    Consider the commutative square
    \[
    \begin{tikzcd}
        \Alg_{\mathbb E_k}(\Mod_R)^{\pi_*\text{-\'et}}_{A//C}
        \arrow[r,"\pi_*"] \arrow[d]
        &
        \left(\CAlg^{\heartsuit}_{\pi_*A}\right)^{\text{\'et}}_{/\pi_*C}
        \arrow[d]\\
        \Alg_{\mathbb E_k}(\Mod_R)^{\pi_*\text{-\'et}}_{A/}
        \arrow[r,"\pi_*"']
        &
        \CAlg^{\heartsuit,\text{\'et}}_{\pi_*A}.
    \end{tikzcd}
    \]
    The vertical maps are right fibrations, classified respectively by the
    functors
    \[
        B\longmapsto
        \Map_{\Alg_{\mathbb E_k}(\Mod_R)_{A/}}(B,C)
    \]
    and
    \[
        B_0\longmapsto
        \Map_{\CAlg^{\heartsuit}_{\pi_*A}}(B_0,\pi_*C).
    \]
    The bottom horizontal functor is an equivalence by Theorem
    \ref{underrigidity}, and Proposition \ref{ffness} identifies the
    corresponding fibers.  The fiberwise criterion for right fibrations
    \cite[Corollary~2.4.4.4]{HTT} therefore shows that the top horizontal
    functor is an equivalence.
\end{proof}

\begin{corollary}\label{cotangentvanishing}
    Let $f\colon A \rightarrow B$ be a $\pi_*$-\'etale extension of $\mathbb{E}_k$-$R$-algebras for some $\mathbb{E}_{k+1}$-ring spectrum $R$, $2 \leq k \leq \infty$. Then the relative cotangent complex $\mathbb{L}_{B/A}^{\mathbb{E}_k} \in \Mod^{\E_k}_B(\LMod_R)$ vanishes.
\end{corollary}

\begin{proof}
    Let $M\in \Mod^{\E_k}_B(\LMod_R)$ and form the square-zero extension $B\oplus M\to B$. A derivation from $B$ to $M$ over $A$ is equivalently a lift of the map $A\to B\oplus M$ through $B\oplus M\to B$ whose composite with $B\oplus M\to B$ is the identity of $B$. Apply Theorem \ref{k+1} to the over-category with target $C=B\oplus M$. Since $A\to B$ is $\pi_*$-\'etale, the space of such lifts is identified with the corresponding discrete lifting space for the \'etale Dirac morphism $\pi_*A\to\pi_*B$ against the square-zero extension
    \[
        \pi_*B\oplus\pi_*M\longrightarrow \pi_*B.
    \]
    Formal \'etaleness of $\pi_*A\to\pi_*B$ makes this discrete lifting space a point. Therefore
    \[
        \Map_{\Mod^{\E_k}_B(\LMod_R)}(\mathbb L^{\E_k}_{B/A},M)
        \simeq *
    \]
    for every $M$, and hence $\mathbb L^{\E_k}_{B/A}\simeq 0$.
\end{proof}

\section{The \texorpdfstring{$\mathbb E_4$}{E4}-orientation of the Lubin--Tate theory}\label{sec:BHLS-application}

Burklund, Hahn, Levy, and Schlank construct an $\mathbb E_3$-$MU_{(p)}$-algebra structure on the Lubin--Tate spectrum and promote it to an $\mathbb E_4$-orientation
\[
    MU_{(p)}\longrightarrow E_n
\]
in \cite[Proposition~5.9, Construction~5.10, and Corollary~5.12]{BHLS}.  We construct the $\mathbb E_3$-$MU_{(p)}$-algebra structure by completed relative obstruction theory, then obtain the $\mathbb E_4$-orientation using their self-centrality theorem \cite[Proposition~5.11]{BHLS}.

The coefficient map from a truncated Brown--Peterson spectrum to Lubin--Tate theory is not $\pi_*$-\'etale before completion.  Its relative cotangent complex nevertheless becomes zero after completion at
\[
    I_n=(p,v_1,\ldots,v_{n-1}).
\]
The coefficient modules in the Goerss--Hopkins tower are $I_n$-complete, so the obstruction spaces can be computed from the completed relative cotangent complex.

\subsection{Completed obstruction spaces}\label{subsec:completedPV}

For a Dirac ring $B_0$ and a finitely generated homogeneous ideal $I\subset B_0$, we use derived $I$-completion in graded $B_0$-modules \cite[Definition~7.3.1.1 and Notation~7.3.1.5]{SAG}. To apply the constructions of \cite[Section~7.3]{SAG} to graded modules, we represent multiplication by a homogeneous element $x$ of degree $d$ by $M[-d]\to M$. Iterated multiplication includes the corresponding iterated grading shifts, and the diagrams and their limits are formed in graded modules. Throughout this section, we use these conventions when applying the completion formalism in the stable $B_0$-linear category of operadic modules.

\begin{lemma}\label{lem:completionoperadic}
    Let $B_0$ be a Dirac ring, let $I\subset B_0$ be a finitely generated homogeneous ideal, and let
    \[
        \Mod^{\mathbb E_k}_{B_0}
    \]
    denote the stable category of operadic $B_0$-modules.  The full subcategory of derived $I$-complete objects is stable, and its inclusion admits an exact left adjoint
    \[
        M\longmapsto M_I^{\wedge}.
    \]
    Consequently, for every derived $I$-complete operadic module $N$, the unit $M\to M_I^{\wedge}$ induces a canonical equivalence
    \[
        \Map_{\Mod^{\mathbb E_k}_{B_0}}(M_I^{\wedge},N)
        \simeq
        \Map_{\Mod^{\mathbb E_k}_{B_0}}(M,N).
    \]
\end{lemma}

\begin{proof}
    With the grading conventions above, the construction of \cite[Proposition~7.3.1.4]{SAG} gives a semiorthogonal decomposition of $\Mod^{\mathbb E_k}_{B_0}$ into $I$-local and $I$-complete objects. The projection onto the complete subcategory is left adjoint to its inclusion. Since both categories are stable, this left adjoint is exact. The mapping-space equivalence is the adjunction formula.
\end{proof}

For a map $A_0\to B_0$ of Dirac algebras and a finitely generated homogeneous ideal $I\subset A_0$, put
\[
    \widehat{\mathbb L}^{\mathbb E_k}_{B_0/A_0}
    :=
    \left(\mathbb L^{\mathbb E_k}_{B_0/A_0}\right)^{\wedge}_{I}
    \in \Mod^{\mathbb E_k}_{B_0},
\]
where $I$ also denotes the ideal generated by its image in $B_0$.

If $N$ is a derived $I$-complete operadic $B_0$-module, Lemma \ref{lem:completionoperadic} gives
\begin{equation}\label{eq:completed-obstruction-spaces}
    \Map_{\Mod^{\mathbb E_k}_{B_0}}
    \left(\mathbb L^{\mathbb E_k}_{B_0/A_0},N\right)
    \simeq
    \Map_{\Mod^{\mathbb E_k}_{B_0}}
    \left(\widehat{\mathbb L}^{\mathbb E_k}_{B_0/A_0},N\right).
\end{equation}
Thus, after the reduction to stage zero in Lemma \ref{relativebasechange}, the object and morphism obstruction spaces of Lemmas \ref{relobjob} and \ref{relmorob} may be computed from the completed relative cotangent complex. If $\widehat{\mathbb L}^{\mathbb E_k}_{B_0/A_0}\simeq0$, every such obstruction and ambiguity space with derived $I$-complete coefficients is contractible.

\begin{lemma}\label{lem:relativePVconvergence}
    Let $R$ be an $\mathbb E_{k+1}$-ring spectrum and work in $\operatorname{Syn}_R$. Let $A$ be an $\mathbb E_k$-$R$-algebra, and let $A_n$ be its compatible system of potential stages. For a fixed map of $0$-stages $A_0\to B_0$, there is a natural equivalence
    \[
    \begin{split}
        &\Alg_{\mathbb E_k}(\mathcal M_\infty)_{A/}
        \times_{\Alg_{\mathbb E_k}(\mathcal M_0)_{A_0/}}
        \{B_0\}
        \\
        &\qquad\simeq
        \varprojlim_n\left(
            \Alg_{\mathbb E_k}(\mathcal M_n)_{A_n/}
            \times_{\Alg_{\mathbb E_k}(\mathcal M_0)_{A_0/}}
            \{B_0\}
        \right).
    \end{split}
    \]
\end{lemma}

\begin{proof}
    By Construction \ref{SynthR}, the category $\operatorname{Syn}_R$ is complete. Hence \cite[Remark~5.2]{pstragowski2021abstract} identifies the category of potential $\infty$-stages with the limit of the finite-stage categories:
    \[
        \Alg_{\mathbb E_k}(\mathcal M_\infty)
        \simeq
        \varprojlim_n \Alg_{\mathbb E_k}(\mathcal M_n).
    \]
    Passing to the slice under the compatible system $A_\bullet$ and then taking the fiber over the fixed $0$-stage $B_0$ commute with limits in $\Cat_\infty$. This gives the claimed equivalence.
\end{proof}

\begin{proposition}\label{prop:completedPVexistence}
    Let $R$ be an $\mathbb E_{k+1}$-ring spectrum and let $A$ be a fixed $\mathbb E_k$-$R$-algebra, with $A_0=\pi_*A$. Let $A_0\to B_0$ be a map of Dirac algebras and let $I\subset A_0$ be a finitely generated homogeneous ideal. Assume:
    \begin{enumerate}
        \item $B_0$, and hence every shift of $B_0$ occurring as an obstruction coefficient, is derived $I$-complete;
        \item the completed relative cotangent complex vanishes:
        \[
            \widehat{\mathbb L}^{\mathbb E_k}_{B_0/A_0}\simeq 0.
        \]
    \end{enumerate}
    Then the space
    \[
        \Alg_{\mathbb E_k}(\mathcal M_\infty)_{A/}
        \times_{\Alg_{\mathbb E_k}(\mathcal M_0)_{A_0/}}
        \{B_0\}
    \]
    is contractible. Equivalently, the space of $\mathbb E_k$-$R$-algebras $B$ under $A$, equipped with a specified identification
    \[
        \pi_*B\simeq B_0
    \]
    compatible with the prescribed map $A_0\to B_0$, is contractible. In particular, such a realization exists.
\end{proposition}

\begin{proof}
    Start with the prescribed $0$-stage $A_0\to B_0$. By Lemmas \ref{relobjob}, \ref{relmorob}, and \ref{relativebasechange}, the object and morphism obstructions at each subsequent stage lie in mapping spaces out of $\mathbb L^{\mathbb E_k}_{B_0/A_0}$. The coefficient modules are derived $I$-complete, so \eqref{eq:completed-obstruction-spaces} and the vanishing hypothesis make these mapping spaces contractible. Both objects and morphisms therefore lift with contractible spaces of choices. Induction shows that the fiber over $B_0$ at every finite stage is nonempty and contractible. By Lemma \ref{lem:relativePVconvergence}, the realization space is the inverse limit of these fibers and is contractible.
\end{proof}

\subsection{The completed cotangent calculation}\label{subsec:completed-cotangent}

We now show that the relative cotangent complex of $A\to\widehat A$ vanishes after derived $I$-completion.

\begin{lemma}\label{lem:completionformal}
    Let $A$ be a Noetherian Dirac ring, let $I\subset A$ be a finitely generated homogeneous ideal, and let
    \[
        A\longrightarrow \widehat A
    \]
    be its derived $I$-completion.  Assume that $\widehat A$ is discrete as a graded derived ring.  Then, for every $1\leq k\leq\infty$,
    \[
        \left(
            \mathbb L^{\mathbb E_k}_{\widehat A/A}
        \right)^{\wedge}_{I}
        \simeq 0
    \]
    in $\Mod^{\mathbb E_k}_{\widehat A}$.
\end{lemma}

\begin{proof}
    Equip the complete module category with the tensor product obtained by completing the ordinary tensor product. The completion functor is monoidal for this structure \cite[Variant~7.3.5.6]{SAG}, with the grading conventions of Subsection~\ref{subsec:completedPV}. Its induced adjunction on $\mathbb E_k$-algebras makes $\widehat A$ initial among derived $I$-complete $\mathbb E_k$-$A$-algebras.

    Let $N$ be a derived $I$-complete operadic $\widehat A$-module. The square-zero extension $\widehat A\oplus N$ is derived $I$-complete, since the complete subcategory is stable.  The universal property of completion implies that both mapping spaces
    \[
        \Map_{\Alg_{\mathbb E_k,A/}}(\widehat A,\widehat A\oplus N)
        \quad\text{and}\quad
        \Map_{\Alg_{\mathbb E_k,A/}}(\widehat A,\widehat A)
    \]
    are contractible.  The projection $\widehat A\oplus N\to\widehat A$ induces a map between these spaces, and its fiber over the identity of $\widehat A$ is contractible because both spaces are contractible.  By the universal property of the operadic cotangent complex, this fiber is
    \[
        \operatorname{Der}^{\mathbb E_k}_{A}(\widehat A,N)
        \simeq
        \Map_{\Mod^{\mathbb E_k}_{\widehat A}}
        \left(
            \mathbb L^{\mathbb E_k}_{\widehat A/A},N
        \right).
    \]
    Lemma \ref{lem:completionoperadic} now gives
    \[
        \Map_{\Mod^{\mathbb E_k}_{\widehat A}}
        \left(
            \left(\mathbb L^{\mathbb E_k}_{\widehat A/A}\right)^{\wedge}_{I},N
        \right)
        \simeq *
    \]
    for every complete $N$. Taking
    \[
        N=\left(\mathbb L^{\mathbb E_k}_{\widehat A/A}\right)^{\wedge}_{I}
    \]
    shows that its identity map is nullhomotopic, and hence that this object is zero.
\end{proof}

\subsection{Application to Lubin--Tate theory}\label{subsec:lubinTateCalculation}

Fix the $\mathbb E_3$-$MU_{(p)}$-algebra form of $\mathrm{BP}\langle n\rangle$ supplied by \cite[Theorem~A and Remark~1.0.11]{truncBP}.  Write
\[
    A_*:=\pi_*\mathrm{BP}\langle n\rangle
    \simeq \mathbb Z_{(p)}[v_1,\ldots,v_n],
    \qquad
    I=I_n=(p,v_1,\ldots,v_{n-1}).
\]
We use the coordinate on the universal deformation of the Honda height-$n$ formal group for which
\[
    E_*:=\pi_*E_n
    \simeq
    W(\mathbb F_{p^n})[[u_1,\ldots,u_{n-1}]][u^{\pm1}],
    \qquad |u|=-2,
\]
and the induced coefficient map is
\[
    v_i\longmapsto u_i u^{1-p^i}\quad (1\leq i<n),
    \qquad
    v_n\longmapsto u^{1-p^n}.
\]
We apply the obstruction theory to this coefficient map and the chosen $\mathbb E_3$-$MU_{(p)}$-algebra form of $\mathrm{BP}\langle n\rangle$.

Let
\[
    \widehat A_*
    :=
    \left(A_*[v_n^{-1}]\right)^{\wedge}_{I}
    \simeq
    \varprojlim_{r\geq1} A_*[v_n^{-1}]/I^r,
\]
where the inverse limit is taken degreewise in graded rings. The ring $A_*[v_n^{-1}]$ is Noetherian, so the comparison with classical completion \cite[Corollary~7.3.6.6]{SAG}, applied with the grading conventions of Subsection~\ref{subsec:completedPV}, identifies its derived completion with this discrete inverse limit.

\begin{proposition}\label{prop:MoravaEAdicEtale}
    The coefficient map factors as
    \[
        A_*\longrightarrow A_*[v_n^{-1}]
        \longrightarrow \widehat A_*
        \longrightarrow E_*,
    \]
    where $\widehat A_*\to E_*$ is finite \'etale as a Dirac ring map.  For every $2\leq k\leq\infty$,
    \[
        \left(
            \mathbb L^{\mathbb E_k}_{E_*/A_*}
        \right)^{\wedge}_{I}
        \simeq 0.
    \]
\end{proposition}

\begin{proof}
    Put $W=W(\mathbb F_{p^n})$ and $d=p^n-1$, and let
    \[
        P_*:=\mathbb Z_p[v_1,\ldots,v_{n-1},v_n^{\pm1}],
        \qquad
        T_*:=W[u_1,\ldots,u_{n-1}][u^{\pm1}],
    \]
    where $|u_i|=0$ and $|u|=-2$. Give $T_*$ the $P_*$-algebra structure
    \[
        v_i\longmapsto u_i u^{1-p^i}\quad (1\leq i<n),
        \qquad
        v_n\longmapsto u^{-d}.
    \]
    The substitutions $u_i=v_i u^{p^i-1}$ give an isomorphism
    \[
        T_*\simeq
        (P_*\otimes_{\mathbb Z_p}W)[u^{\pm1}]/(u^d-v_n^{-1}).
    \]
    Since $W/\mathbb Z_p$ is finite unramified and $d\,u^{d-1}$ is a unit,
    the map $P_*\to T_*$ is finite \'etale. Moreover, $T_*$ is finite free
    as a graded $P_*$-module, with a homogeneous basis.

    The map $A_*[v_n^{-1}]\to P_*$ induces isomorphisms modulo $I^r$
    for every $r\geq1$, since $p^r\in I^r$. Hence
    \[
        \widehat A_*\simeq\varprojlim_r P_*/I^r.
    \]
    Finite graded freeness allows tensoring with $T_*$ to commute with
    this degreewise inverse limit, giving
    \[
        \widehat A_*\otimes_{P_*}T_*
        \simeq\varprojlim_r T_*/I^rT_*.
    \]
    Here $IT_*=(p,u_1,\ldots,u_{n-1})$. Since $W$ is $p$-adically
    complete and the $u_i$ have degree zero, the degreewise completion is
    \[
        \varprojlim_r T_*/I^rT_*
        \simeq W[[u_1,\ldots,u_{n-1}]][u^{\pm1}]
        =E_*.
    \]
    Thus $\widehat A_*\to E_*$ is the base change of the finite \'etale
    map $P_*\to T_*$, and is therefore finite \'etale.

    The localization $A_*\to A_*[v_n^{-1}]$ is an \'etale map of Dirac rings, so
    \[
        \mathbb L^{\mathbb E_k}_{A_*[v_n^{-1}]/A_*}\simeq0
    \]
    for $k\geq2$ by \cite[Proposition~4.35]{DG1}.  The transitivity sequence and Lemma \ref{lem:completionformal} therefore give
    \[
        \left(
            \mathbb L^{\mathbb E_k}_{\widehat A_*/A_*}
        \right)^{\wedge}_{I}
        \simeq0.
    \]

    Apply transitivity once more to
    \[
        A_*\longrightarrow \widehat A_*\longrightarrow E_*.
    \]
    It gives a cofiber sequence
    \[
        E_*\boxtimes_{\widehat A_*}
        \mathbb L^{\mathbb E_k}_{\widehat A_*/A_*}
        \longrightarrow
        \mathbb L^{\mathbb E_k}_{E_*/A_*}
        \longrightarrow
        \mathbb L^{\mathbb E_k}_{E_*/\widehat A_*}.
    \]
    Apply derived $I$-completion to this sequence; it remains a cofiber sequence by Lemma \ref{lem:completionoperadic}. Compatibility of completion with tensor products \cite[Proposition~7.3.5.1]{SAG} identifies the completion of the first term with that of
    \[
        E_*\boxtimes_{\widehat A_*}
        \left(
            \mathbb L^{\mathbb E_k}_{\widehat A_*/A_*}
        \right)^{\wedge}_{I},
    \]
    and this vanishes.  The last term vanishes before completion because $\widehat A_*\to E_*$ is \'etale, again by \cite[Proposition~4.35]{DG1}.  The completed middle term is therefore zero.
\end{proof}

\begin{corollary}\label{cor:LubinTateE3MU}
    The space of $I_n$-complete $\mathbb E_3$-$MU_{(p)}$-algebras $B$ under $\mathrm{BP}\langle n\rangle$, equipped with an identification $\pi_*B\simeq E_*$ compatible with the standard coefficient map $A_*\to E_*$, is contractible. Such a realization is equivalent to $E_n$ as an $\mathbb E_3$-ring, by an equivalence preserving the complex orientation. Transport gives an $\mathbb E_3$-$MU_{(p)}$-algebra map
    \[
        \mathrm{BP}\langle n\rangle\longrightarrow E_n
    \]
    inducing the standard coefficient map.
\end{corollary}

\begin{proof}
    Proposition \ref{prop:MoravaEAdicEtale} gives
    \[
        \left(
            \mathbb L^{\mathbb E_3}_{E_*/A_*}
        \right)^{\wedge}_{I}
        \simeq0.
    \]
    The ring $E_*$ and all of its grading and suspension shifts are derived $I$-complete. Proposition \ref{prop:completedPVexistence}, applied in
    \[
        \Alg_{\mathbb E_3}(\Mod_{MU_{(p)}}),
    \]
    therefore shows that the space of realizations under the fixed source $\mathrm{BP}\langle n\rangle$, equipped with the specified identification of homotopy groups with $E_*$, is contractible.

    The $MU_*$-algebra $E_*$ is Landweber exact, so the coefficient identification is realized by an isomorphism $B\simeq E_n$ of homotopy $MU_{(p)}$-algebras \cite[Theorem~2.8 and Proposition~2.20]{HoveyStrickland}. By the uniqueness of the $\mathbb E_3$-structure on the Lubin--Tate homotopy algebra \cite[Corollary~4.53]{Ramzi2023}, this identification lifts to an $\mathbb E_3$-equivalence. Transport along this equivalence gives the asserted map with its complex orientation.
\end{proof}

Thus the existence statement of \cite[Construction~5.10]{BHLS} is recovered from the completed relative obstruction theory. However, we do not address whether the resulting $\mathbb E_3$-$MU_{(p)}$-algebra structure is equivalent to the one constructed there.

Combining this construction with the self-centrality equivalence of \cite[Proposition~5.11]{BHLS}, we recover the $\mathbb E_4$-orientation stated in \cite[Corollary~5.12]{BHLS}.

\begin{corollary}\label{cor:LubinTateE4orientation}
    There exists an $\mathbb E_4$-ring map
    \[
        MU_{(p)}\longrightarrow E_n
    \]
    whose underlying complex orientation is the standard Lubin--Tate orientation.
\end{corollary}

\begin{proof}
    By the universal property of the $\mathbb E_3$-center, the $\mathbb E_3$-$MU_{(p)}$-algebra structure constructed above determines an $\mathbb E_4$-map
    \[
        MU_{(p)}\longrightarrow Z_{\mathbb E_3}(E_n).
    \]
    The self-centrality theorem \cite[Proposition~5.11]{BHLS} identifies the canonical map
    \[
        Z_{\mathbb E_3}(E_n)\longrightarrow E_n
    \]
    as an equivalence of $\mathbb E_4$-algebras.  Composing gives the required $\mathbb E_4$-orientation.
\end{proof}

\appendix
\section{An alternative proof by Dunn additivity}
\label{app:dunn-additivity}

We deduce the case $k\geq2$ of Theorem~\ref{k+1} from its case
$k=1$ by Dunn additivity, following \cite[Proof of Theorem~7.5.4.2]{HA}.

Fix an $\mathbb E_{k+1}$-ring spectrum $R$ and a $\pi_*$-Dirac
$\mathbb E_k$-$R$-algebra $C$.  The category
$\Alg_{\mathbb E_1}(\Mod_R)$ is $\mathbb E_{k-1}$-monoidal.  For a
homogeneous class $x\in\pi_dB$, represented by
$\widetilde x\colon \Sigma^dR\to B$, let
\[
    \lambda_x,\rho_x\colon \Sigma^dB\longrightarrow B
\]
be graded left and right multiplication:
\[
\begin{split}
    \lambda_x&\colon \Sigma^dR\otimes_RB
        \xrightarrow{\widetilde x\otimes\mathrm{id}}
        B\otimes_RB\xrightarrow{\mu_B}B,\\
    \rho_x&\colon \Sigma^dR\otimes_RB
        \xrightarrow{\beta_{\Sigma^dR,B}}
        B\otimes_R\Sigma^dR
        \xrightarrow{\mathrm{id}\otimes\widetilde x}
        B\otimes_RB\xrightarrow{\mu_B}B.
\end{split}
\]

\begin{definition}\label{def:good-prime}
    An $\mathbb E_1$-$R$-algebra $B$ is \emph{good$'$} if
    $\lambda_x\simeq\rho_x$ for every homogeneous $x\in\pi_*B$.
\end{definition}

This is the graded analogue of Lurie's notion of a good algebra. The next
lemma corresponds to observations~(i)--(iii) in
\cite[Proof of Theorem~7.5.4.2]{HA}.

\begin{lemma}\label{lem:good-prime-properties}
    \begin{enumerate}
        \item The underlying $\mathbb E_1$-algebra of an
        $\mathbb E_2$-$R$-algebra is good$'$.
        \item If $B$ is good$'$, then $\pi_*B$ is a Dirac ring.
        \item Good$'$ $\mathbb E_1$-$R$-algebras are closed under tensor
        products.
    \end{enumerate}
\end{lemma}

\begin{proof}
    The first assertion follows by restricting the interchange homotopy along
    $\widetilde x\otimes\mathrm{id}_B$.  For the second, evaluating
    $\lambda_x\simeq\rho_x$ on $y\in\pi_eB$ gives
    \[
        xy=(-1)^{de}yx.
    \]

    The third assertion is the graded counterpart of observation~(iii) in
    \cite[Proof of Theorem~7.5.4.2]{HA}. A homogeneous class of degree $d$
    is represented by a map from the invertible object $\Sigma^dR$.
    The double braiding of $\Sigma^dR$ with any $R$-module is canonically
    homotopic to the identity, while a single braiding records the Koszul
    sign. The tensor-product argument then identifies graded left and right
    multiplication by every homogeneous class of the tensor product.
\end{proof}

We also need the graded analogue of \cite[Lemma~7.5.4.8]{HA}.

\begin{lemma}
\label{lem:graded-etale-tensor}
    Let $f\colon B\to B'$ be a $\pi_*$-\'etale morphism of $\pi_*$-Dirac
    $\mathbb E_1$-$R$-algebras, and let $D$ be another
    $\mathbb E_1$-$R$-algebra.  If $B\otimes_RD$ is $\pi_*$-Dirac, then
    \[
        f\otimes\mathrm{id}_D\colon
        B\otimes_RD\longrightarrow B'\otimes_RD
    \]
    is $\pi_*$-\'etale, and
    \[
        \pi_*(B'\otimes_RD)
        \simeq
        \pi_*B'\otimes_{\pi_*B}\pi_*(B\otimes_RD).
    \]
\end{lemma}

\begin{proof}
    There is a canonical equivalence of underlying modules
    \[
        B'\otimes_RD\simeq B'\otimes_B(B\otimes_RD).
    \]
    Since $\pi_*B'$ is flat over $\pi_*B$, the K\"unneth spectral sequence
    for extension of scalars collapses and gives the displayed isomorphism.
    The induced homotopy-ring map is the base change of
    $\pi_*B\to\pi_*B'$, and is therefore \'etale.  In particular,
    $B'\otimes_RD$ is $\pi_*$-Dirac.
\end{proof}

Let $\mathcal C$ be the full subcategory of
\[
    \bigl(\Alg_{\mathbb E_1}(\Mod_R)\bigr)_{/C}
\]
spanned by the maps $B\to C$ for which $B$ is good$'$.  By
Lemma~\ref{lem:good-prime-properties}, $\mathcal C$ is closed under tensor
products and inherits an $\mathbb E_{k-1}$-monoidal structure.

Let $\mathcal D$ be the full subcategory of
\[
    \Fun\!\left(
        \Delta^1,
        \bigl(\Alg_{\mathbb E_1}(\Mod_R)\bigr)_{/C}
    \right)
\]
spanned by the $\pi_*$-\'etale maps $B\to B'$ whose source is good$'$.
The pointwise tensor product makes $\mathcal D$ an
$\mathbb E_{k-1}$-monoidal category.  Indeed, if
$f_i\colon B_i\to B_i'$ belong to $\mathcal D$, then
\[
    f_0\otimes f_1
    =
    (\mathrm{id}_{B_0'}\otimes f_1)
    \circ
    (f_0\otimes\mathrm{id}_{B_1}).
\]
The source $B_0\otimes_RB_1$ is good$'$, and two applications of
Lemma~\ref{lem:graded-etale-tensor} show that the composite is
$\pi_*$-\'etale.

The functor
\[
    \pi_*\colon \mathcal C\longrightarrow
    (\CAlg^{\heartsuit}_{\pi_*R})_{/\pi_*C}
\]
is lax $\mathbb E_{k-1}$-monoidal.  Form the fiber product of
$\infty$-operads
\[
    (\mathcal C')^{\otimes}
    :=
    \mathcal C^{\otimes}
    \times_{((\CAlg^{\heartsuit}_{\pi_*R})_{/\pi_*C})^{\otimes}}
    \Fun^{\mathrm{\acute et}}\!\left(
        \Delta^1,
        (\CAlg^{\heartsuit}_{\pi_*R})_{/\pi_*C}
    \right)^{\otimes},
\]
where the second map is evaluation at the source.  An object of
$\mathcal C'$ is a map $B\to C$ with $B$ good$'$, together with an \'etale triangle
\[
    \pi_*B\longrightarrow D_*\longrightarrow\pi_*C.
\]
Its tensor product is
\[
\begin{split}
 &(B_0,\pi_*B_0\to D_{0,*})\otimes
   (B_1,\pi_*B_1\to D_{1,*})\\
 &\quad\simeq
 \left(
    B_0\otimes_RB_1,
    \pi_*(B_0\otimes_RB_1)
    \longrightarrow
    D_{0,*}\otimes_{\pi_*B_0}
    \pi_*(B_0\otimes_RB_1)
    \otimes_{\pi_*B_1}D_{1,*}
 \right).
\end{split}
\]
Thus the construction uses the actual homotopy ring of the tensor product
and does not require $\pi_*$ to be strong monoidal.

The assignment
\[
    \theta\colon \mathcal D\longrightarrow\mathcal C',
    \qquad
    (B\to B')\longmapsto(B,\pi_*B\to\pi_*B')
\]
is $\mathbb E_{k-1}$-monoidal by
Lemma~\ref{lem:graded-etale-tensor}.  We claim that it is an equivalence.
There is a commutative diagram
\[
\begin{tikzcd}
    \mathcal D \ar[rr,"\theta"] \ar[dr,"p"']
        && \mathcal C' \ar[dl,"q"]\\
    & \mathcal C,
\end{tikzcd}
\]
where $p$ and $q$ are cocartesian fibrations and $\theta$ carries
$p$-cocartesian morphisms to $q$-cocartesian morphisms.  Over a fixed
map $B\to C$ with $B$ good$'$, the induced functor on fibers is
\[
    \Alg_{\mathbb E_1}(\Mod_R)^{\pi_*\text{-\'et}}_{B//C}
    \longrightarrow
    \left(\CAlg^{\heartsuit}_{\pi_*B}\right)^{\mathrm{\acute et}}_{/\pi_*C},
\]
which is an equivalence by the case $k=1$ of Theorem~\ref{k+1}.  Hence
$\theta$ is an equivalence on underlying categories by
\cite[Corollary~2.4.4.4]{HTT}, and therefore an equivalence of
$\mathbb E_{k-1}$-monoidal categories by
\cite[Remark~2.1.3.8]{HA}.

\begin{proposition}
\label{prop:good-prime-dunn-proof}
    For $k\geq2$, Theorem~\ref{k+1} follows from its case $k=1$ by Dunn
    additivity.
\end{proposition}

\begin{proof}
    Dunn additivity, Lemma~\ref{lem:good-prime-properties}(1), the monoidal
    equivalence $\theta$, and the cocartesian monoidal structures on the
    algebraic categories give the following sequence of equivalences:
    \[
    \begin{aligned}
    &\Fun^{\pi_*\text{-\'et}}\!\left(
        \Delta^1,
        \bigl(\Alg_{\mathbb E_k}(\Mod_R)\bigr)_{/C}
    \right)\\
    &\quad\simeq \Alg_{\mathbb E_{k-1}}(\mathcal D)\\
    &\quad\simeq \Alg_{\mathbb E_{k-1}}(\mathcal C')\\
    &\quad\simeq
    \Alg_{\mathbb E_{k-1}}(\mathcal C)
    \times_{\Alg_{\mathbb E_{k-1}}((\CAlg^{\heartsuit}_{\pi_*R})_{/\pi_*C})}
    \Alg_{\mathbb E_{k-1}}\!\left(
        \Fun^{\mathrm{\acute et}}\!\left(
            \Delta^1,
            (\CAlg^{\heartsuit}_{\pi_*R})_{/\pi_*C}
        \right)
    \right)\\
    &\quad\simeq
    \bigl(\Alg_{\mathbb E_k}(\Mod_R)\bigr)_{/C}
    \times_{(\CAlg^{\heartsuit}_{\pi_*R})_{/\pi_*C}}
    \Fun^{\mathrm{\acute et}}\!\left(
        \Delta^1,
        (\CAlg^{\heartsuit}_{\pi_*R})_{/\pi_*C}
    \right).
    \end{aligned}
    \]
    Taking the fiber over $A\to C$ identifies the left-hand side with
    $\Alg_{\mathbb E_k}(\Mod_R)^{\pi_*\text{-\'et}}_{A//C}$ and the
    right-hand side with
    $\left(\CAlg^{\heartsuit}_{\pi_*A}\right)^{\mathrm{\acute et}}_{/\pi_*C}$.
\end{proof}

\end{document}